\documentclass[11pt,reqno]{amsart}
\usepackage{latexsym}
\usepackage{amssymb}
\usepackage{amsmath}
\usepackage{amscd}
\usepackage{amsthm}
\usepackage{graphicx}
\usepackage[hang]{caption}
\usepackage{indentfirst}
\usepackage{dsfont}
\usepackage{color}
\usepackage{bm}
\usepackage[section]{placeins}

\newcommand\dela[1]{}
\allowdisplaybreaks

\def\XXint#1#2#3{{\setbox0=\hbox{$#1{#2#3}{\int}$ }
\vcenter{\hbox{$#2#3$ }}\kern-.6\wd0}}

\newtheorem{theorem}{Theorem}
\numberwithin{equation}{section}
\numberwithin{theorem}{section}
\newtheorem{lemma}[theorem]{Lemma}
\newtheorem{cor}[theorem]{Corollary}

\newtheorem{definition}[theorem]{Definition}

\numberwithin{equation}{section}

\newcommand{\mf}{\mathbf}

\newcommand{\wt}{\widetilde}
\newcommand{\e}{\varepsilon}
\def\R{\mathbb{ R}}
\def\Z{\mathbb{ Z}}
\def\P{\mathbb{ P}}
\def\E{\mathbb{E}}
\def\P1{\mathbb{P}^{1}}

\def\o{\overline}
\begin{document}


\title{Homogenization of the stochastic Navier--Stokes equation 
with a stochastic slip boundary condition}

\author[Hakima  Bessaih]{Hakima Bessaih}
\address{University of Wyoming, Department of Mathematics, Dept. 3036, 1000
East University Avenue, Laramie WY 82071, United States}
\email{ bessaih@uwyo.edu}

\author[Florian  Maris]{ Florian  Maris}
\address{Numerical Porous Media SRI Center, CEMSE Division,
King Abdullah University of Science and Technology, 
Thuwal 23955-6900,
Kingdom of Saudi Arabia}
\email{florinmaris@gmail.com}

\maketitle


{\footnotesize
\begin{center}


\end{center}
}
\begin{abstract}
The two dimensional Navier-Stokes equation in a perforated domain with a dynamical slip boundary condition is considered. We assume that the dynamic is driven by a stochastic perturbation on the interior of the domain and another stochastic perturbation on the boundaries of the holes.  We consider a scaling ($\e^2$ for the viscosity and 1 for the density) that will lead to a time dependent limit problem. However, the noncritical scaling  ($\e^\beta,\quad \beta>1$) is considered in front of the nonlinear term.  The homogenized system in the limit is obtained as  a Darcy's law with memory with two permeabilities and an extra term that is due to the stochastic perturbation on the boundary of the holes. 
We use the two-scale convergence method.
Due to the stochastic integral, the pressure that appears in the variational formulation does not have enough regularity in time. This fact made us rely only on the variational formulation for the passage to the limit on the solution. We obtain a variational formulation for the limit that is solution of a Stokes system with two pressures.  This two-scale limit gives rise to three cell problems, two of them give the permeabilities while the third one  gives an extra term in the Darcy's law due to the stochastic perturbation on the boundary of the holes.
 
 \end{abstract}

\maketitle

\section{Introduction and formulation of the problem}
\label{section1}
In this paper we are interested in the asymptotic behavior of a two-dimensional  Navier Stokes equation 
in a domain containing periodically distributed obstacles subject to a dynamical slip boundary condition. We assume that the dynamic is driven by a stochastic perturbation on the interior of the domain and another stochastic perturbation on the boundaries of the obstacles. We represent the solid obstacles by holes in the fluid domain.

We consider $D$, an open bounded  domain of $\R^2$ with a Lipschitz boundary 
$\partial D$. Let $Y=[0, 1[ ^2$ the representative cell and denote by $O$ an open subset of $Y$ with a smooth boundary $\partial O$, such that $\o{O}\subset Y$, and set $Y^*=Y\setminus \o{O}$. The elementary cell $Y$ and the small cavity or hole $O$ inside it are used to model small scale obstacles or heterogeneities in a physical medium $D$. Denote by $O^{\e,k}$ the translation of $\e O$ by $\e k$, $k\in\Z^{2}$. We make the assumption that the holes do not intersect the boundary $\partial D$ and we denote by $\mathcal{K^\e}$ the set of all $k\in\Z^2$ such that the cell $\e k +\e Y$ is strictly included in $D$. The set of all such holes will be denoted by $O^{\e}$, i.e.
$$O^{\e}:=\bigcup_{k\in\mathcal{K^\e}}O^{\e,k}=\bigcup_{k\in\mathcal{K^\e}}\e(k+O),$$
and set
$$D^\e:=D\setminus \o{O^{\e}}.$$
By this construction, $D^\e$ is a periodically perforated domain with holes of size of the same order as the period. More preciselly, we are interested in the following system of equations:
\begin{equation}
\label{systemu}
\left\{
\begin{array}{rll}
d u^\e(t,x) =&\left [ \nu\e^2\Delta u^\e (t,x) + \e^\beta (u^\e(t,x)\cdot \nabla ) u^\e(t,x)- \nabla p^\e(t,x)\right ]dt\\
&+f(t,x) dt+g_1(t) d W_1(t,x) &\mbox{in}\ [0,T] \times D^\e, \\
\operatorname{div} u^\e(t,x) =&0&\mbox{in}\  [0,T]\times D^\e, \\
u^\e(t,x) \cdot n=&0&\mbox{on}\ [0,T]\times\partial O^\e, \\
\e d u^\e_\tau(t,x) =&-\left[\nu\e^2 \left( \dfrac{\partial u^\e(t,x)}{\partial n}\right)_\tau  + \alpha^\e(x) u^\e_\tau(t,x)\right] dt+\e \left[g_2^\e(t)\right]_\tau d W_2(t,x)&\mbox{on}\  [0,T]\times\partial O^\e, \\
u^\e(0,x)=&u_0^\e(x)&\mbox{in}\   D^\e, \\
u^\e(0,x)=&v_0^\e (x)&\mbox{on}\  \partial O^\e, \\
\end{array}
\right.
\end{equation}
where $p^\e$ is the pressure of the fluid, $\nu > 0$ is the viscosity, $u^\e$ is the velocity, $n$ is the normal vector at the boundary $\partial O^\e$ that points inside the hole, $u^\e_\tau$ is the tangential velocity on $\partial O^\e$, and $\left( \dfrac{\partial u^\e(t,x)}{\partial n}\right)_\tau$ is the tangential component of the normal derivative of the velocity.

Here $(W_1(t))_{t\geq 0}$ and $(W_2(t))_{t\geq 0}$ are two mutually independent 
$L^2(D)^2$-- valued Wiener processes defined on the complete probability space $(\Omega,\mathcal{F}, \mathbb{P})$ endowed with the canonical filtration $(\mathcal{F}_t)_{t\geq 0}$ and with the covariances $Q_1$ and $Q_2$, where $Q_1$ and $Q_2$ are two linear positive operators in  
$L^2(D)^2$ of trace class.
Also $f\in L^2([0,T]\times D)^2$ represents the body forces, $\alpha^\e$ is a strictely positive function in $L^\infty (\partial O^\e)$ and $\beta >1$. For any $K$ and $H$ two separable Hilbert spaces, we denote by $L_2(K,H)$ the space of bounded linear operators that are Hilbert-Schmidt from $K$ to $H$. If $Q$ is a linear positive operator in $K$ of trace class, then we denote by $L_Q(K,H)$, the space of bounded linear operators that are Hilbert-Schmidt from $Q^{\frac{1}{2}} K$ to $H$, and the norm is denoted by $\|\cdot\|_Q$. Finally $g_1\in C([0,T]; L_{Q_{1}}(L^2(D)^2, L^2(D)^2)$ and $g_2^\e\in C([0,T]; L_{Q_{2}}(L^2(D)^2, L^2(\partial O^\e)^2)$, with 
\begin{equation}
\label{defg2e}
g_2^\e(t)=g_{21}(t)+\mathcal{R}^\e g_{22}(t),
\end{equation}
where $g_{21}\in C([0,T]; L_{Q_{2}}(L^2(D)^2,H^1(D)^2)$, $g_{22}\in C([0,T]; L_{Q_{2}}(L^2(D)^2,L^2(\partial O)^2)$. For any element $h \in L^2(\partial O)$, we define $\mathcal{R}^\e h \in L^2(\partial O^\e)$ by 
\begin{equation}
\label{defRe}
\mathcal{R}^\e h(x) = h\left(\dfrac{x}{\e}\right)
\end{equation}
where $h$ is considered $Y-$ periodic.

With a similar scaling, but in a deterministic and linear setting, and with Dirichlet boundary conditions on the boundary of the holes, Allaire in \cite{A91} proved rigorously the convergence of the homogenization process to a Darcy's law with memory. Let us notice that the scaling $\e^2$ for the viscosity is the precise scaling that gives a non-zero limit for $u^\e$ when $\e\to 0$ for this particular geometry. Moreover, the scaling $1$ for the density of the fluid is the precise one that gives a time-dependent limit problem. Indeed, Mikeli{\'c} in \cite{Mik91} studied the deterministic Navier Stokes equation, with Dirichlet boundary condition on the boundaries of the holes, but with a scaling $\e$ for the density, and obtained in the limit a stationary Darcy's law. The Stokes equation in a perforated domain, with a slip boundary condition has been studied by Allaire in \cite{A95} where Darcy's law was obtained in the limit with a permeability that depends on the slip coefficient $\alpha$. Similar results can be found in \cite{CDE96} with a nonhomogeneous slip boundary condition, that translates in an additional term in the Darcy's law. The stochastic Stokes with a stochastic dynamical Fourier boundary condition was analyzed in 
\cite{maris}, where in the limit, a  stochastic parabolic partial differential equation was obtained. In the present paper we analyze the full stochastic Navier Stokes equation, but with a noncritical scaling $\e^\beta$, with $\beta>1$ in front of the nonlinear term. Our dynamical boundary condition is of slip type, and contains a scaling of $\e$, which is the precise one that gives in the limit a Darcy's law with memory with a second permeability. The nonhomogeinity on the boundary contains a periodic stochastic part, that yields in the limit to an additional term in the Darcy's law.

To prove the convergence to the homogenized problem, we use the  two-scale convergence method that goes back to \cite{N89,A-2s} and was extended to the stochastic case in \cite{BMW94}. Because our uniform estimates are in mean due to the presence of the random coefficient $\omega$, the two-scale 
convergence method has to be adapted to our setting. More precisely, we extend the notion of two-scale convergence to contain more parameters, $\omega$, $t$ and $x$. This is very similar to what we  previously introduced in our paper \cite{maris}.
 We extend by $0$ inside the holes the velocities $u^\e$ and $\nabla u^\e$ and we pass to the limit in the two scale convergence sense in the variational formulation. We obtain a variational formulation for the limit $u^*$ which  is associated with a stochastic Stokes system with two pressures. The existence of the two pressures was proven by the orthogonality Lemma \ref{ortog2}. This two scale limit gives rise to three cell problems, two of them give the permeabilities while the third cell problem gives  an extra term in the Darcy's law. Due to the stochastic integral, we have to define a more regular in time pressure $P(t)$ that appears in the variational formulation. For this reason, in order to show Darcy's law we had to rely only on the variational formulation, and this is proven in Theorem \ref{thform}. 
 
The idea of the two scale convergence method was to give a rigorous justification to the asymptotic expansion method. A more general setting has been defined by Nguetseng in \cite{N03}, \cite{N04} and later in \cite{NSW10}. The theory of the two scale convergence from the periodic to the stochastic setting has been extended by Bourgeat, A. Mikeli{\'c} and Wright in \cite{BMW94}, using techniques from ergodic theory. There is a vast literature for partial differential equations with random coefficients, where this method was used, however most of the tackled problems in this setting are not in perforated domains, (see \cite{BMW94}, \cite{BLL06} and the references therein). Much less was done for homogenization of stochastic partial differential equations, in particular in perforated domains. We mention the paper \cite{WD07} where a reaction diffusion equation with a dynamical boundary condition with a noise source term on both the interior of the domain and on the boundary was studied, and through a tightness argument and a pointwise two scale convergence method the homogenized equation was derived. A comprehensive theory for solving stochastic homogenization problems has been constructed recently in \cite{JL13} and \cite{RSW12}, where a $\Sigma$-convergence method adapted to stochastic processes was developed. An application of the method to the homogenization of a stochastic Navier-Stokes type equation with oscillating coefficients in a bounded domain (without holes) has been provided.

For the deterministic Stokes or Navier Stokes equations in perforated domains we refer to: \cite{SP80}, \cite{T80}, \cite {A92}, \cite{Mik91}, \cite{CDE96}, \cite{A91}, \cite{Mik95}. In \cite{CDE96} the Stokes problem in a perforated domain with a nonhomogeneous Fourier boundary condition on the boundaries of the holes was studied while in \cite{A91} the same problem was studied with a slip boundary condition. The stochastic Stokes in a perforated domain has been studied first in our previous  paper in \cite{maris}. As far as we know, this is the first result on the stochastic Navier-Stokes in perforated domain, although we are studying the noncritical case.

The paper is organized as follows. In Section \ref{section2} we define the functional spaces and rewrite the system in an abstract form. In Section \ref{section3} we study the microscopic equation, we prove the existence and uniqueness of solutions, and get the uniform estimates. In Section \ref{section4} the two scale convergence results are introduced and adapted to our particular setting. In Section \ref{section5},  we derive all the two-scale limits and pass to the limit in the variational formulation using particular test functions.   We obtain the variational formulation for the two scale limit $u^*$. The homogenized system is studied in section 6 and an explicit formula in the form of a Darcy's law with memory is given in theorem \ref{thform}. Section 7 is dedicated to some concluding remarks about Darcy's law obtained in the previous section.

\section{Functional setting}
\label{section2}
Let us introduce the following Hilbert spaces

\begin{equation}
\label{spaceLe}
\mf{L}^2_\e:=L^2(D^\e)^2\times L^2(\partial O^\e)^2,
\end{equation}

\begin{equation}
\label{spaceHe}
\mf{H}^1_\e:=H^1(D^\e)^2\times H^{\frac{1}{2}}(\partial O^\e)^2.
\end{equation}
equipped respectively with the inner products

\[
  \langle \mf{U},\mf{V}\rangle=\int_{D^\e} \big[  u(x) \cdot v(x)\big] dx 
  +\int_{\partial O^{\e}} \big[  \o{u}(x') \cdot \o{v}(x')\big] d\sigma(x'),
\]
and 
\[
 ((\mf{U},\mf{V})) = \int_{D^\e} \big[  u(x) \cdot v(x)\big] dx + \int_{D^\e} \big[ \nabla u(x) \cdot\nabla v(x)\big] dx,\]
where $\mf{U}=\left(\begin{array}{c} 
u\\
\o{u}
\end{array}\right)$.

We introduce the bounded linear and surjective operator
$\gamma^\e:H^1(D^\e)^2\mapsto H^{\frac{1}{2}}(\partial O^\e)^2$ such that
$\gamma^{\e}u=u|_{\partial O^\e}$  for all $u\in C^\infty\left(\overline{D^\e}\right)^2$. $\gamma^\e$ is the  trace operator, (see \cite{sohr}, pp. 47). We denote by $H^{-\frac{1}{2}}(\partial O^\e)^2$ the dual space of 
$H^{\frac{1}{2}}(\partial O^\e)^2$.

We denote by $\mf{H}^\e$ the closure of $\mathcal{V}^\e$ in $\mf{L}^2_\e$, and by $\mf{V}^\e$ the closure of $\mathcal{V}^\e$ in $\mf{H}^1_\e$,  
where 
\begin{equation}
\label{spaceVe}
\begin{split}
\mathcal{V}^\e:=\left\{\mf{U}=\left(\begin{array}{c} 
u\\
\o{u}
\end{array}\right)
\in C^\infty\left(\overline{D^\e}\right)^2\times \gamma^\e(C^\infty\left(\overline{D^\e}\right)^2)\ \ |\ \ \mbox{div }u=0,\ \o{u}=\sqrt{\e} u_{\partial O^\e},\right.\\
\left. u=0 \mbox{ on } \partial D, \overline{u}\cdot n =0 \mbox{ on } \partial O^\e\right\},
\end{split}
\end{equation}

Let $\Pi^\e: \mf{L}^2_\e \mapsto L^2(D^\e)^2$ be the operator that represents the projection onto the first component, i.e. $\Pi^\e \mf{U} =u$, for every $\mf{U}=(u,\o{u})\in \mathbf{L}^2_\e$. We also define the spaces $H^\e=\Pi^\e \mf{H}^\e$ and $V^\e=\Pi^\e\mf{V}^\e$.

$\mf{H}^\e$ and $\mf{V}^\e$ are separable Hilbert spaces with the inner products and norms inherited from $\mf{L}^2_\e$ and $\mf{H}^1_\e$ respectively:
\[
\|\mf{U}\|^2_{\mf{H}^\e}=\langle \mf{U},\mf{U}\rangle\, , 
\]
\[
 \|\mf{U}\|_{\mf{V}^\e}^2 = ((\mf{U},\mf{U})) \, , 
\]
and $H^\e$ and $V^\e$ are also separable Hilbert spaces with the norms induced by the projection $\Pi^\e$.

Denoting by $(\mf{H}^\e )'$ and $(\mf{V}^\e )'$ the dual spaces, if we identify  $\mf{H}^\e$ with  $(\mf{H}^\e) '$
then we have the Gelfand triple $\mf{V}^\e\subset \mf{H}^\e \subset (\mf{V}^\e)'$ with continuous injections. 

We denote the dual pairing between $\mf{U}\in \mf{V}^\e$ and $\mf{V}\in (\mf{V}^\e)'$ by
$\langle  \mf{U},\mf{V}\rangle_{\langle(\mf{V}^\e)',\mf{V}^\e\rangle}$.
When $\mf{U}\in \mf{H}^\e$, we have $\langle \mf{U},\mf{V}\rangle_{\langle(\mf{V}^\e)',\mf{V}^\e\rangle}=\langle \mf{U},\mf{V}\rangle$.
Define the linear operator $\mf{A}^\e : D(\mf{A}^\e) \subset \mf{H}^\e\mapsto (\mf{H}^\e)'$:

\begin{equation}
\label{defAe}
\mf{A}^\e \mf{U}=\mf{A}^\e \left( \begin{array}{c}
u\\
\o{u}
\end{array}\right)
=
\operatorname{Proj}_{\mf{H}^{\e}}\left(\begin{array}{c}
-\nu  \Delta u\\
\dfrac{\nu} {\sqrt{\e}}  \left(\dfrac{\partial u}{\partial n}\right)_\tau+\dfrac{\alpha^\e}{\e^3}\o{u}_\tau
\end{array}\right),
\end{equation}
with
$$D(\mf{A}^\e)=\{ \mf{U}\in \mf{V}^\e  | -\Delta u \in L^2(D^\e)^2\ \mbox{ and } \left(\dfrac{\partial u}{\partial n}\right)_\tau \in L^2(\partial O^\e)^2\},$$
and
$$\langle\mf{A}^\e \mf{U} , \mf{V}\rangle=\displaystyle\int_{D^\e} \nu \nabla u \nabla v dx+\int_{\partial O^\e} \frac{\alpha^\e}{\e^3} \overline{u}_\tau \overline{v}_\tau d \sigma.$$

\begin{lemma}\label{lemmaAe} (Properties of the operator $\mathbf{A}^\e$) For every $\e > 0$, the linear operator $\mf{A}^\e$ is positive and self-adjoint in $\mf{H}^\e$.
\end{lemma}
\begin{proof}
The operator is obviously symmetric, since for every $\mf{U},\ \mf{V} \in \mf{V}^\e$,
$$\langle \mf{A}^\e \mf{U}, \mf{V}\rangle =\nu \displaystyle\int_{D^\e} \nabla u \nabla v dx+\int_{\partial O^\e} \frac{\alpha^\e}{\e^2} \gamma^\e(u)\gamma^\e (v) d \sigma$$
and also coercive by using the strict positivity of $\alpha^\e$ and the Poincare inequality
$$\langle \mf{A}^\e \mf{U}, \mf{U}\rangle = \nu \displaystyle\int_{D^\e} \nabla u \nabla u dx+\int_{\partial O^\e} \frac{\alpha^\e}{\e^2} \gamma^\e(u)\gamma^\e (u) d \sigma \geq c(\e) ||\mf{U}||^2_{\mf{V}^\e}.$$

To show that it is self-adjoint it is enough to show that 
$$D(\mf{A}^\e) = \{\mf{U} \in \mf{V}^\e\ |\  \langle \mf{A}^\e \mf{U} ,\mf{V} \rangle \leq C ||\mf{V}||_{\mf{H}^\e} \mbox{ for every } \mf{V} \in \mf{V}^\e\}.$$

But, $\langle \mf{A}^\e \mf{U}, \mf{V}\rangle \leq C ||\mf{V}||_{\mf{H}^\e}$ for every $\mf{V} \in \mf{V}^\e$ is equivalent to
\begin{equation} 
\begin{split}
\nu \displaystyle\int_{D^\e} \nabla u \nabla v dx \leq C||v||_{L^2(D^\e)^2} + C ||v||_{L^2(\partial O^\e)^2}\ \forall v\in V^\e \Longleftrightarrow\\
\nu \displaystyle\int_{D^\e} -\Delta u  v dx + \displaystyle \nu\int_{\partial O^\e} \dfrac{\partial u}{\partial n} v d\sigma \leq C||v||_{L^2(D^\e)^2} + C ||v||_{L^2(\partial O^\e)^2}\ \forall v\in V^\e \Longleftrightarrow\\
\end{split}
\end{equation}
$\Delta u \in L^2(D^\e)^2$ and $\dfrac{\partial u}{\partial n} \in L^2(\partial O^\e)^2$, so $\mf{U} \in D(\mf{A}^\e)$.
Now we use Proposition A.10, page 389 from \cite{DPZ} to infer that $\mf{A}^\e$ is self-adjoint and generates an analytic  semigroup.
\end{proof}
By continuity, the operator $\mf{A}^{\e}$ can be extended from $\mf{V}^\e$ into $( \mf{V}^\e)'$. 

Denote  by $S_{\e}(t)$  the analytic semigroup generated by $\e^2\mf{A}^{\e}$ (see \cite{pazy}).

Let $\mf{B}^\e: \mf{V}^\e \times \mf{V}^\e \to (\mf{V}^\e)'$, and $B^\e: V^\e \times V^\e \to (V^\e)'$ be defined by
\begin{equation}
\label{defBe}
\mf{B}^\e \left(\mf{U},\mf{V}\right)=
\left(
\begin{array}{c}
B^\e(u,v)\\
0
\end{array}
\right)
=
\left(
\begin{array}{c}
(u\cdot \nabla) v\\
0
\end{array}
\right).
\end{equation}

\begin{lemma}\label{propBe}(cf. \cite{temam})
  Let $u,v, z\in V^\e$. Then
  \begin{equation} \label{prop1Be} \langle B^\e(u,v),z\rangle_{(V^\e)',V^\e}=- \langle
    B^\e(u,z),v\rangle_{(V^\e)',V^\e} \quad \mbox{\rm and}\quad \langle
    B^\e(u,v),v\rangle_{(V^\e)',V^\e} =0.
  \end{equation}
  Furthermore,
  \begin{equation}\label{prop2Be}
    |\langle B^\e(u,v), z\rangle_{(V^\e)',V^\e}|
		\le \|u\|_{L^{4}(D^\e)^2}\|v\|_{V^\e}\|z\|_{L^{4}(D^\e)^2}.
  \end{equation}
\end{lemma}
Moreover, one can apply the two-dimensional Ladyzhenskaya
interpolation inequality (cf. \cite{constantin1988})
  \begin{equation}\label{inter}
    \|u\|^{2}_{L^{4}(D^\e)^2}\leq C_{\e}\|u\|_{H^\e}\|u\|_{V^\e},
  \end{equation}
  to the right-hand side of \eqref{prop2Be} to obtain
\begin{equation}\label{prop3Be}
|\langle B^\e(u,v), z\rangle_{(V^\e)',V^\e}| \leq C_{\e}
\|u\|_{H^\e}^{1/2}\|u\|_{V^\e}^{1/2} \|v\|_{V^\e}\|z\|_{H^\e}^{1/2}\|z\|_{V^\e}^{1/2}.
\end{equation}

If we define $\mathbf{F}^\e \in L^2(0,T;L^2(D^\e)^2\times L^2(\partial O^\e))$ by
\begin{equation}
\label{defFe}
\mathbf{F}^\e(t,x)=\left(\begin{array}{c}
f(t,x)\\
0
\end{array}\right),
\end{equation}
and

\begin{equation}
\label{defGe}
\mathbf{G}^\e(t)=\left(\begin{array}{cc}
g_1(t) & 0\\
0 & \e \left[g_{2}^\e(t)\right]_\tau
\end{array}\right), \qquad  \mathbf{W}(t)=(W_{1}(t), W_{2}(t)),
\end{equation}
we can rewrite the system \eqref{systemu} in the compact form

\begin{equation}
\label{systemU}
\left\{
\begin{array}{rll}
d \mf{U}^\e(t) &+\e^2\mathbf{A}^\e \mf{U}^\e (t) dt + \e^\beta \mf{B}^\e(\mf{U}^\e(t),\mf{U}^\e(t)) dt=\mathbf{F}(t) dt+\mathbf{G}^\e(t) d \mathbf{W}, \\

\mf{U}^\e(0)&=\mf{U}_0^\e=\left( \begin{array}{c}
u_0^\e\\
\overline{u_{0}^{\e}}=\sqrt{\e} v_0^\e
\end{array}
\right).
\end{array}
\right.
\end{equation}


We assume that the operators $g_1$, $g_{21}$, and $g_{22}$ satisfy the properties:

\begin{equation}
\label{eqCT}
\begin{split}
||g_1(t)||^2_{Q_1} := 
\sum_{j=1}^\infty \lambda_{j1}||g_1(t) e_{j1}||^2_{L^2(D)^2} \leq C_T, \ t\in [0,T],\\
||g_{21}(t)||^2_{Q_2} := 
\sum_{j=1}^\infty \lambda_{j2}||g_{21}(t) e_{j2}||^2_{H^1(D)^2} \leq C_T, \ t\in [0,T],\\
||g_{22}(t)||^2_{Q_2} := 
\sum_{j=1}^\infty \lambda_{j2} ||g_{22}(t) e_{j2}||^2_{L^2(\partial O)^2} \leq C_T, \ t\in [0,T],\\
\end{split}
\end{equation}
where $\{e_{j1}\}_{j=1}^\infty$ and $\{e_{j2}\}_{j=1}^\infty$ are respectively the eigenvectors of  $Q_1$ and $Q_2$, and $\{\lambda_{j1}\}_{j=1}^\infty$ and $\{\lambda_{j2}\}_{j=1}^\infty$ are the corresponding sequences of eigenvalues.

Moreover, throughout the paper we will assume that 
$\mf{U}_{0}^{\e}= (u^{\e}_{0},\overline{u^{\e}_{0}}=\sqrt{\e} v_0^\e)$ is an $\mathcal{F}_{0}-$ measurable $\mf{V}^\e-$ valued random variable and there exists a constant $C$ independent of $\e$, such that for every $\e > 0$:

\begin{equation}\label{u0}
\mathbb{E}\|u^{\e}_{0}\|_{L^2(D^\e)}^2+\e^2\mathbb{E}\|\nabla u_0^\e\|_{L^{2}(\partial O^\e)}^2\leq C.
\end{equation}

For any function $z^\e$ defined in $D^\e$ we will denote in this paper by $\wt{z}^{\e}$ the extension of $z^{\e}$ by $0$ to $D$,
\begin{equation}
\wt{z}^{\e}:=\left\{\begin{array}{lr}
z^{\e} &{\rm on}\ D^\e\\
0 & {\rm on} \ D\setminus D^\e.
\end{array}
\right.
\end{equation} 
\section{The microscopic model}
\label{section3}

\subsection{Auxiliary Stokes type problem}
\label{subsection21}

We introduce the associated linear Ornstein--Uhlenbeck process:

\begin{equation}
\label{systemZ}
\left\{
\begin{array}{rll}
d \mf{Z}^\e(t) &+\e^2\mathbf{A}^\e \mf{U}^\e (t) dt=\mathbf{G}^\e(t) d \mathbf{W} \\
\\
\mf{Z}^\e(0)&=\mf{0}.
\end{array}
\right.
\end{equation}

\begin{theorem}\label{exZ}
For any $T>0$, the system \eqref{systemZ} has a unique mild solution 
$\mf{Z}^\e\in L^{2}(\Omega, C([0,T], \mf{H}^\e)\cap L^{2}(0,T; \mf{V}^\e))$, 

\begin{equation}\label{mildZ} 
\mf{Z}^\e(t)=\int_{0}^{t}S_{\e}(t-s)\mathbf{G}^\e(s)d\mathbf{W}, \quad t\in [0,T].
\end{equation}
The mild solution $\mf{Z}^\e$ is also a weak solution, that is,  $\mathbb{P}$-a.s.

\begin{equation}\label{varZ}
\langle \mf{Z}^\e(t), \bm{\phi}\rangle+\e^2\int_{0}^{t}\langle (\mf{A}^\e)^{1/2} \mf{Z}^\e(s),(\mf{A}^\e)^{1/2}\bm{\phi}\rangle ds
=\int_{0}^{t}\langle \mf{G}^\e(s)d\mf{W}(s),\bm{\phi}\rangle,
\end{equation}
for $t\in [0,T]$ and $\bm{\phi}\in \mf{V}^\e$. 

Moreover, for every $\e>0$

\begin{equation}\label{estZ1}
\mathbb{E}\|\mf{Z}^\e(t)\|_{\mf{H}^\e}^{2}
+\e^2\mathbb{E}\int_{0}^{t}\|\mf{Z}^\e(s)\|_{\mf{V}^\e}^{2}ds
\leq C_{T}, \quad t\in [0,T]
\end{equation}
and 

\begin{equation}\label{estZ2}
\mathbb{E}\sup_{t\in[0,T]}\|\mf{Z}^\e(t)\|_{\mf{H}^\e}^{2}
\leq C_{T}.
\end{equation}

\end{theorem}

\begin{proof}
Since the operator $ \e^2 \mf{A}^{\e}$ is the generator of a strongly continuous semigroup 
$S_{\e}(t),\ t\geq 0$ in $\mf{H}^\e$ and using the assumption \eqref{eqCT}, then the existence and uniqueness of mild solutions in  $\mf{H}^\e$ is a consequence of Theorem 7.4 of \cite{DPZ}.
The regularity in  $\mf{V}^\e$ is a consequence of the estimates below.

Applying the It\^{o} formula to $\|\mf{Z}^\e(t)\|_{\mf{H}^\e}^{2}$, we get that
\begin{align*}
d\|\mf{Z}^\e(t)\|_{\mf{H}^\e}^{2}&=2\langle \mf{Z}^\e(t),d\mf{Z}^\e(t)\rangle dt+\| \mf{G}^\e(t)\|_{Q}^{2}dt\\
&=-2\langle \e^2\mf{A}^\e \mf{Z}^\e(t),\mf{Z}^\e(t)\rangle dt
+2\langle \mf{G}^\e(t)d\mf{W},\mf{Z}^\e(t)\rangle dt
+\| \mf{G}^\e(t)\|_{Q}^{2}dt.
\end{align*}
Hence,

\begin{align}\label{Z1}
\|\mf{Z}^\e(t)\|_{\mf{H}^\e}^{2}&+2\e^2\int_{0}^{t}\|(\mf{A}^\e)^{1/2} \mf{Z}^\e(s)\|_{\mf{H}^\e}^{2}ds\leq 
\|\mf{Z}^\e_{0}\|_{\mf{H}^\e}^{2}+\int_{0}^{t}\|\mf{Z}^\e(s)\|_{\mf{H}^\e}^{2}ds
\nonumber\\
&+2\int_{0}^{t} \langle \mf{G}^\e(s)d\mf{W},\mf{Z}^\e(t)\rangle +\int_{0}^{t}\| \mf{G}^\e(s)\|_{Q}^{2}ds.
\end{align}
Taking the expected value yields

\begin{equation}\label{Z2}
\mathbb{E}\|\mf{Z}^\e(t)\|_{\mf{H}^\e}^{2}\leq
\int_{0}^{t}\mathbb{E}\|\mf{Z}^\e(s)\|_{\mf{H}^\e}^{2}ds
+\int_{0}^{t}\| \mf{G}^\e(s)\|_{Q}^{2}ds
\end{equation}
and 
\begin{equation}\label{Z3}
\e^2\mathbb{E}\int_{0}^{t}\|\mf{Z}^\e(s)\|_{\mf{V}^\e}^{2}ds
\leq \int_{0}^{t}\mathbb{E}\|\mf{Z}^\e(s)\|_{\mf{H}^\e}^{2}ds 
+\int_{0}^{t}\| \mf{G}^\e(s)\|_{Q}^{2}ds
\end{equation}

Now  \eqref{estZ1} follows from using Gronwall's lemma in \eqref{Z2}.

On the other side, \eqref{Z1} implies that
\begin{align*}
\sup_{0\leq t\leq T}\|\mf{Z}^\e(t)\|_{\mf{H}^\e}^{2}&+
2\e^2\int_{0}^{T}\|\mf{Z}^\e(s)\|_{\mf{V}^\e}^{2}ds\leq 
\int_{0}^{T}\|\mf{Z}^\e(s)\|_{\mf{H}^\e}^{2}ds\\
&+2\sup_{0\leq t\leq T}\left |\int_{0}^{t} \langle \mf{G}^\e(s)d\mf{W},\mf{Z}^\e(t)\rangle \right |
+\int_{0}^{T}\| \mf{G}^\e(s)\|_{Q}^{2}ds
\end{align*}
Moreover using the Burkholder-David-Gundy inequality and the Young inequality, we get that

\begin{align*}
\mathbb{E}\sup_{0\leq t\leq T}\left |\int_{0}^{t} \langle \mf{G}^\e(s)d\mf{W},\mf{Z}^\e(t)\rangle \right | &\leq 
\mathbb{E}\left(\int_{0}^{T}\left |\mf{G}^\e(s) \mf{Z}^\e(s)\right |^{2}ds\right)^{1/2}\\
&\leq \frac{1}{2} \mathbb{E}\sup_{0\leq t\leq T}\|\mf{Z}^\e(t)\|_{\mf{H}^\e}^{2}
+ C\int_{0}^{T}|\mf{G}^\e(s)|^{2}ds\\
&\leq \frac{1}{2} \mathbb{E}\sup_{0\leq t\leq T}\|\mf{Z}^\e(t)\|_{\mf{H}^\e}^{2}+C_{T}
\end{align*}

Now, plugging this estimate in the previous one and using Gronwall's lemma completes the proof of 
\eqref{estZ2}.
\end{proof}

\subsection{Well posedness of the stochastic Navier Stokes equation}
\label{subsection22}
In this paragraph we will state the existence and uniqueness of solutions for the 2-d stochastic Navier Stokes equation with a dynamical slip boundary condition, coupled with a Dirichlet boundary condition driven by a noise on the boundary and interior of the domain. Since the noise is additive, we will use a pathwise argument. For similar results with different boundary conditions, see \cite{flandoli1994}.

\begin{theorem}\label{exU}
For any $\e > 0$, $\mf{U}_0^\e \in \mf{H}^\e$ and $T>0$, there exists a unique stochastic process $\mf{U}^\e$, solution of equation \eqref{systemU} in the following
sense: $\mathbb{P}$-a.s.
$$\mf{U}^\e\in C([0,T]; \mf{H}^\e)\cap L^{2}([0,T]; \mf{V}^\e)$$
and
\begin{align}\label{varU}
\langle \mf{U}^\e(t),\bm{\phi}\rangle &+\e^2\int_{0}^{t}\langle (\mf{A}^\e)^{1/2} \mf{U}^\e(s),(\mf{A}^\e)^{1/2}\bm{\phi}\rangle ds
-\e^\beta\int_{0}^{t} \langle \mf{B}^\e(\mf{U}^\e(s), \bm{\phi}), \mf{U}^\e(s)\rangle ds=\langle \mf{U}^\e_{0},\bm{\phi}\rangle\nonumber \\
&+\int_{0}^{t} \langle \mf{F}(s),\bm{\phi}\rangle ds + \int_{0}^{t} \langle  \mf{G}^\e(s)  d\mf{W}(s), \bm{\phi}\rangle
\end{align}
for all $t\in [0,T]$ and for all $\bm{\phi}\in \mf{V}^\e$.
Moreover,
\begin{equation}\label{estU}
 \E\left(\sup_{0\leq t\leq T}\|\mf{U}^\e(t)\|_{\mf{H}^\e}^{2}
	+\e^2\int_{0}^{T}\|\mf{U}^\e(t)\|_{\mf{V}^\e}^{2}dt \right)<C_T.
\end{equation}
\end{theorem}
\begin{proof}
The proof of this theorem is based on a pathwise argument. Let us denote by $\o{\mf{U}}^\e=\mf{U}^\e - \mf{Z}^\e$, then $\o{\mf{U}}^\e$ is the solution of the following random differential equation
\begin{equation}
\label{systemoU}
\left\{
\begin{array}{rll}
d \o{\mf{U}}^\e(t) &+\e^2\mathbf{A}^\e \o{\mf{U}}^\e (t) dt + \e^\beta \mf{B}^\e(\o{\mf{U}}^\e(t) +\mf{Z}^\e(t),\o{\mf{U}}^\e(t)+\mf{Z}^\e(t))dt=\mathbf{F}(t) dt \\

\o{\mf{U}}^\e(0)&=\left( \begin{array}{c}
u_0^\e\\
\overline{u_{0}^{\e}}=\sqrt{\e} v_0^\e
\end{array}
\right),
\end{array}
\right.
\end{equation}
that we study as a deterministic evolution equation for almost every $\omega\in\Omega$. More precisely, we show that for almost every $\omega\in\Omega$ the solution $\o{\mf{U}}$ of the equation \eqref{systemoU} exists, $\o{\mf{U}} \in C([0,T]; \mf{H}^\e)\cap L^{2}([0,T]; \mf{V}^\e)$, satisfies the weak formulation
\begin{align}\label{varoU}
\langle \o{\mf{U}}^\e(t),\bm{\phi}\rangle &+\e^2\int_{0}^{t}\langle (\mf{A}^\e)^{1/2} \o{\mf{U}}^\e(s),(\mf{A}^\e)^{1/2}\bm{\phi}\rangle ds
-\e^\beta\int_{0}^{t} \langle \mf{B}^\e(\o{\mf{U}}^\e(s) + \mf{Z}(s), \bm{\phi}), \o{\mf{U}}^\e(s)+\mf{Z}(s)\rangle ds\\
&=\langle \o{\mf{U}}^\e_{0},\bm{\phi}\rangle+\int_{0}^{t} \langle \mf{F}(s),\bm{\phi}\rangle ds \bm{\phi}\rangle,
\end{align}
for all $t\in [0,T]$ and for all $\bm{\phi}\in \mf{V}^\e$.

First we prove some a priori estimates for the solution of \eqref{systemoU}. We have for almost all $t\in [0,T]$:
\begin{equation}
\begin{split}
&\left\langle(\o{\mf{U}}^\e)'(t) , \o{\mf{U}}^\e(t) \right\rangle_{\langle(\mf{V}^\e)',\mf{V}^\e\rangle} + \e^2 \left\langle\mf{A}^\e \o{\mf{U}}^\e(t) , \o{\mf{U}}^\e(t) \right\rangle_{\langle(\mf{V}^\e)',\mf{V}^\e\rangle}\\
&+\e^\beta \left\langle\mf{B}^\e (\o{\mf{U}}^\e(t)+\mf{Z}^\e(t),\o{\mf{U}}^\e(t)+\mf{Z}^\e)(t),\o{\mf{U}^\e}(t) \right\rangle_{\langle(\mf{V}^\e)',\mf{V}^\e\rangle}=\left\langle\mf{F}(t) , \o{\mf{U}}^\e(t) \right\rangle_{\langle(\mf{V}^\e)',\mf{V}^\e\rangle},
\end{split}
\end{equation}
which implies after using Lemma \ref{propBe} that
\begin{equation}
\begin{split}
&\dfrac{1}{2}\dfrac{\partial}{\partial t}\|\o{\mf{U}}^\e(t)\|^2_{\mf{H}^\e} +\e^2c(\e) \|\o{\mf{U}}^\e(t) \|^2_{\mf{V}^\e} \leq \|\mf{F}(t)\|_{\mf{H}^\e} \|\o{\mf{U}}^\e(t)\|_{\mf{H}^\e} \\
&- \e^\beta\left\langle\mf{B}^\e (\o{\mf{U}}^\e(t)+\mf{Z}^\e(t),\o{\mf{U}}^\e(t)),\mf{Z}^\e(t) \right\rangle_{\langle(\mf{V}^\e)',\mf{V}^\e\rangle}\\
&\leq \dfrac{c(\e) \e^2}{4}\| \o{\mf{U}}^\e(t)\|^2_{\mf{V}^\e} + C(\e)\|\mf{F}(t)\|^2_{\mf{H}^\e}- \e^\beta\left\langle\mf{B}^\e (\o{\mf{U}}^\e(t)+\mf{Z}^\e(t),\o{\mf{U}}^\e(t)),\mf{Z}^\e(t) \right\rangle_{\langle(\mf{V}^\e)',\mf{V}^\e\rangle}.
\end{split}
\end{equation}
The property \eqref{prop3Be} implies that 
\begin{equation}\nonumber
\begin{split}
&\e^\beta|\left\langle\mf{B}^\e (\o{\mf{U}}^\e(t)+\mf{Z}^\e(t),\o{\mf{U}}^\e(t)),\mf{Z}^\e(t) \right\rangle_{\langle(\mf{V}^\e)',\mf{V}^\e\rangle}| \leq \e^\beta|\left\langle\mf{B}^\e (\o{\mf{U}}^\e(t),\o{\mf{U}}^\e(t)),\mf{Z}^\e(t) \right\rangle_{\langle(\mf{V}^\e)',\mf{V}^\e\rangle}|\\
&+\e^\beta|\left\langle\mf{B}^\e (\mf{Z}^\e(t),\o{\mf{U}}^\e(t)),\mf{Z}^\e(t) \right\rangle_{\langle(\mf{V}^\e)',\mf{V}^\e\rangle}|=\e^\beta|\left\langle\mf{B}^\e (\o{\mf{U}}^\e(t),\mf{Z}^\e(t)),\mf{U}^\e(t) \right\rangle_{\langle(\mf{V}^\e)',\mf{V}^\e\rangle}|\\
&+\e^\beta|\left\langle\mf{B}^\e (\mf{Z}^\e(t),\o{\mf{U}}^\e(t)),\mf{Z}^\e(t) \right\rangle_{\langle(\mf{V}^\e)',\mf{V}^\e\rangle}| \leq \e^\beta C_\e\| \o{\mf{U}}^\e(t)\|_{\mf{H}^\e} \| \mf{Z}^\e(t)\|_{\mf{V}^\e}  \| \o{\mf{U}}^\e(t)\|_{\mf{V}^\e}\\
&+ \e^\beta C_\e\| \mf{Z}^\e(t)\|_{\mf{H}^\e} \| \o{\mf{U}}^\e(t))\|_{\mf{V}^\e}  \| \mf{Z}^\e(t)\|_{\mf{V}^\e}\leq  \dfrac{c(\e) \e^2}{4}\| \o{\mf{U}}^\e(t)\|^2_{\mf{V}^\e} +\\
&+C(\e) \left(\| \o{\mf{U}}^\e(t)\|^2_{\mf{H}^\e} \| \mf{Z}^\e(t)\|^2_{\mf{V}^\e} +\| \mf{Z}^\e(t)\|^2_{\mf{H}^\e} \| \mf{Z}^\e(t)\|^2_{\mf{V}^\e}\right).
\end{split}
\end{equation}
We integrate the previous inequality over $[0,T]$.
Now, using Theorem \ref{exZ}, we have that 
$\mathbb{P}-{\rm a.s.}\quad  \mf{Z}^\e\in C([0,T]; \mf{H}^\e)\cap L^{2}([0,T]; \mf{V}^\e).$
Hence,

\begin{equation}\label{4}
\begin{split}
 \|\o{\mf{U}}^\e(t)\|^2_{\mf{H}^\e} &+\e^2c(\e) \int_0^t\|\o{\mf{U}}^\e(s) \|^2_{\mf{V}^\e} ds \leq  \|\o{\mf{U}}^\e(0)\|^2_{\mf{H}^\e} + C(\e) \int_0^t \|\mf{F}(s)\|^2_{\mf{H}^\e} ds \\
&+ C(\e) \left(\int_0^t \| \o{\mf{U}}^\e(s)\|^2_{\mf{H}^\e} 
\| \mf{Z}^\e(s)\|^2_{\mf{V}^\e} ds 
+ \int_0^t \| \mf{Z}^\e(s)\|^2_{\mf{H}^\e} \| \mf{Z}^\e(s)\|^2_{\mf{V}^\e} ds\right)\\
&\leq C(\e,T,\omega) +  \|\mf{F}\|_{L^2(0,T;\mf{H}^\e)}  + \|\o{\mf{U}}^\e(0)\|^2_{\mf{H}^\e} 
+C(\e) \int_0^t \| \o{\mf{U}}^\e(s)\|^2_{\mf{H}^\e} \| \mf{Z}^\e(s)\|^2_{\mf{V}^\e} ds,
\end{split}
\end{equation}

where $C(\e,T,\omega)$ is a constant that is $\mathbb{P}$-a.s. bounded and depends  
on $\e$, $T$ and $\omega$.

Applying Gronwall's Lemma we get the estimate:
\begin{equation}\label{estoU1}
\begin{split}
\sup_{[0,T]}\|\o{\mf{U}}^\e(t)\|^2_{\mf{H}^\e}&\leq \left( C(\e,T,\omega)
+ \|\mf{F}\|_{L^2(0,T;\mf{H}^\e)} + \|\o{\mf{U}}^\e(0)\|^2_{\mf{H}^\e}\right) 
e^{C(\e) \displaystyle \int_0^T \| \mf{Z}^\e(s)\|^2_{\mf{V}^\e} ds}\\
&\leq  C(\e,T,\omega) \left(1+ \|\mf{F}\|_{L^2(0,T;\mf{H}^\e)}
+ \|\o{\mf{U}}^\e(0)\|^2_{\mf{H}^\e} \right).
\end{split}
\end{equation}
Applying the estimate \eqref{estoU1} to \eqref{4} we also obtain that:
\begin{equation}\label{estoU2}
\int_0^T\|\o{\mf{U}}^\e(t)\|^2_{\mf{V}^\e} \leq C(\e,T,\omega) \left(1+ \|\mf{F}\|_{L^2(0,T;\mf{H}^\e)}+\|\o{\mf{U}}^\e(0)\|^2_{\mf{H}^\e} \right).
\end{equation}
We apply the Galerkin approximation to our problem. Let us denote by $\mf{H}_n^\e$ n-dimensional subspace of $\mf{H}^\e$ generated by the first $n$ eigenvectors of the operator $\mf{A}^\e$, $\mf{P}_n^\e$ the projection operator from $\mf{H}^\e$ to $\mf{H}^\e_n$, and let $\o{\mf{U}}_n^\e = \mf{P}_n \o{\mf{U}}^\e$. Then, $\o{\mf{U}}_n^\e$ is the unique global solution of the equation
\begin{equation}
\label{systemoUn}
\left\{
\begin{array}{rll}
d \o{\mf{U}}^\e_n(t) &+\e^2\mathbf{A}^\e \o{\mf{U}}^\e_n (t) dt + \e^\beta \mf{P}_n \mf{B}^\e(\o{\mf{U}}_n^\e(t) +\mf{Z}^\e(t),\o{\mf{U}}_n^\e(t)+\mf{Z}^\e(t))dt=\mf{P}_n \mathbf{F}(t) dt \\
\\
\o{\mf{U}}_n^\e(0)&=\mf{P}_n \mf{U}^\e(0).
\end{array}
\right.
\end{equation}
Moreover, $\o{\mf{U}}_n^\e$ satisfies the estimates \eqref{estoU1} and \eqref{estoU2}, hence $\o{\mf{U}}_n^\e$ is uniformly bounded in $L^\infty (0,T; \mf{H}^\e) \cap L^2(0, T; \mf{V}^\e)$, thus $\o{\mf{U}}_n^\e$ is also bounded in $W^{1,2}(0,T;(\mf{V}^\e)')$. Using a compactness argument, we deduce up to a subsequence that 
$$\o{\mf{U}}_n^\e \to \o{\mf{U}}^\e,$$
strongly in $L^2(0,T; \mf{H}^\e)$. Now it is classical to pass to the limit in the variational formulation for $\mf{U}^\e_n$ (see \cite{temam}) to get \eqref{varoU}. Since $\mf{U}^\e=\o{\mf{U}}^\e+\mf{Z}^\e$, we deduce that $\mf{U}^\e$ satisfies the variational formulation \eqref{varU}. To show the uniqueness of 
$\mf{U}^\e$, let us assume that $\mf{U}^\e_1$ and $\mf{U}^\e_2$ are solutions of the equation \eqref{systemU}, then $\mf{U}_*^\e = \mf{U}^\e_2-\mf{U}^\e_1$ is solution of 
\begin{equation}
\label{systemV}
\left\{
\begin{array}{rll}
d \mf{U}_*^\e(t) &+\e^2\mathbf{A}^\e \mf{U}_*^\e (t) dt + \e^\beta \mf{B}^\e(\mf{U}_*^\e(t),\mf{U}_1^\e(t)) dt + \e^\beta \mf{B}^\e(\mf{U}_2^\e(t),\mf{U}_*^\e(t)) dt=\mf{0}\\
\\
\mf{U}_*^\e(0)&=\mf{0},
\end{array}
\right.
\end{equation}
We multiply with $\mf{U}_*^\e$, and we get as before
$$\dfrac{\partial}{\partial t} \| \mf{U}_*^\e\|^2_{\mf{H}^\e} \leq C(\e) \|\mf{U}_1^\e\|^2_{\mf{V}^\e} \| \mf{U}_*^\e\|^2_{\mf{H}^\e} $$
We apply the Gronwall lemma to get the uniqueness.
Measurability of the process $\mf{U}^\e$ follows from the measurability of the Galerkin approximation. 

Applying the It\^{o} formula to $\|\mf{U}^\e(t)\|_{\mf{H}^\e}^{2}$ and using Lemma \ref{propBe}, we get that
\begin{align*}
&d\|\mf{U}^\e(t)\|_{\mf{H}^\e}^{2}=2\langle \mf{U}^\e(t),d\mf{U}^\e(t)\rangle dt+\| \mf{G}^\e(t)\|_{Q}^{2}dt\\
&=-2\langle \e^2\mf{A}^\e \mf{U}^\e(t),\mf{U}^\e(t)\rangle dt - 2\e^\beta \langle\mf{B}(\mf{U}^\e,\mf{U}^\e) , \mf{U}^\e\rangle
+2\langle \mf{G}^\e(t)d\mf{W},\mf{U}^\e(t)\rangle dt
+\| \mf{G}^\e(t)\|_{Q}^{2}dt\\
&=-2\langle \e^2\mf{A}^\e \mf{U}^\e(t),\mf{U}^\e(t)\rangle dt+2\langle \mf{G}^\e(t)d\mf{W},\mf{U}^\e(t)\rangle dt
+\| \mf{G}^\e(t)\|_{Q}^{2}dt.
\end{align*}
Using similar calculations done for $\mf{Z}^\e$ in Theorem \ref{exZ}, we get \eqref{estU} and this completes the proof.

\end{proof}

As a consequence of theorem \ref{exU}, the variational formulation \eqref{varU} can be rewritten in terms of $u$ solution of system \eqref{systemu}. 

\begin{cor}\label{Pe}
$u$ the solution of  system \eqref{systemu} satisfies the following variational formulation:

\begin{equation}\label{varu}
\begin{split}
\int_{D^\e} \left(u^\e(t)-u^\e_{0}-\int_{0}^{t}f(s)ds -\int_{0}^{t}g_{1}(s)dW_{1}(s)+\int_0^t\e^\beta (u^\e(s)\cdot\nabla) u^\e(s)ds \right) \phi dx=\\
-\int_{D^\e}\int_0^t\nu \e^2\nabla u^\e(s)ds \nabla\phi dx+\int_{\partial O^\e} \left(-\e u^\e(t)+\e v^\e_0 +\int_{0}^{t} \e g_{2}^\e(s)dW_{2}(s) ds
- \int_0^t \alpha^\e u^\e(s)ds\right) \phi d\sigma ,
\end{split}
\end{equation}
$\mathbb{P}$-a.s., and for every $\phi \in V^\e$.
\end{cor}

\section{Two scale convergence}
\label{section4}
We will summarize in this section several results about the two scale convergence that we will use throughout the paper. For the results stated without proofs, see \cite{A-2s}, \cite{A95} or \cite{JL13}. First we establish some notations of spaces of periodic functions. We denote by $C_{\#}^k (Y)$ the space of functions from $C^k(\o{Y})$, that have $Y-$ periodic boundary values. By $L^2_{\#}(Y)$ we understand the closure of $C_{\#}(Y)$ in $L^2(Y)$ and by $H^1_{\#}(Y)$ the closure of $C_{\#}^1(Y)$ in $H^1(Y)$. We denote the restrictions of these spaces of functions to $Y^*$ by $C_{\#}^k (Y^*)$, $L^2_{\#}(Y^*)$, and $H^1_{\#}(Y^*)$.

\begin{definition}
\label{2scdef}
We say that a sequence $u^\e \in L^2(\Omega\times [0,T]\times D)^2$ two-scale converges to $u \in L^2(\Omega\times [0,T]\times D \times Y)^2$, and denote this convergence by
$$u^\e \overset{2-s}{\longrightarrow} u\ \ \ \mbox{ in } \Omega\times [0,T]\times D,$$
if for every $\Psi \in L^2(\Omega\times [0,T]\times D; C_{\#} (Y))^2$ we have
\begin{equation*}
\lim_{\e \to 0} \int_\Omega\int_0^T \int _D u^\e (\omega,t,x) \Psi(\omega,t,x,\frac{x}{\e}) dx dt d\mathbb{P} =\int_\Omega\int_0^T\int_D\int_Y u(\omega,t,x,y) \Psi(\omega,t,x,y) dy dx dt d\mathbb{P}.
\end{equation*}
\end{definition}

\begin{theorem}
\label{2scex}
Assume that the sequence $u^\e$ is uniformly bounded in $L^2(\Omega\times [0,T]\times D)^2$. Then there exists a subsequence $(u^{\e'})_{\e' >0}$ and $u^0 \in L^2(\Omega\times [0,T]\times D \times Y)^2 $ such that $u^{\e'}$ two-scale converges to $u^0$ in $\Omega\times [0,T] \times D$.
\end{theorem}

\begin{cor}
\label{2sc-w}
Assume the sequence $u^\e\in L^2(\Omega\times [0,T]\times D)^2$, two-scale converges to $u^0 \in L^2(\Omega\times [0,T]\times D \times Y)^2$. Then, $u^\e$ converges weakly in $L^2(\Omega\times [0,T]\times D)^2$ to $\displaystyle\int_Y u^0(\omega,t,x,y) dy$.
\end{cor}

We will now define the notion of two scale convergence on periodic surfaces, as introduced in \cite{A95}. 

\begin{definition}
\label{2scdef1}
We say that a sequence $v^\e \in L^2(\Omega\times [0,T]\times \partial O^\e)^2$ two-scale converges to $v \in L^2(\Omega\times [0,T]\times D \times \partial O)^2$, and denote this convergence by
$$v^\e \overset{2-s}{\longrightarrow} v\ \ \ \mbox{ in } \Omega\times [0,T]\times D,$$
if for every $\Psi \in L^2(\Omega\times [0,T]; C(\overline{D};C_{\#} (Y))^2$ we have
\begin{equation*}
\begin{split}
\lim_{\e \to 0} &\int_\Omega\int_0^T \int _{\partial O^\e} \e v^\e (\omega,t,x) \Psi(\omega,t,x,\frac{x}{\e}) d\sigma(x) dt d\mathbb{P} \\
&=\int_\Omega\int_0^T\int_D\int_{\partial O} v(\omega,t,x,y) \Psi(\omega,t,x,y) d\sigma(y) dx dt d\mathbb{P}.
\end{split}
\end{equation*}
\end{definition}

\begin{theorem}
\label{2scex1}
Assume that the sequence $\sqrt{\e}v^\e$ is uniformly bounded in $L^2(\Omega\times [0,T]\times \partial O^\e)^2$. Then there exists a subsequence $(v^{\e'})_{\e' >0}$ and $v^0 \in L^2(\Omega\times [0,T]\times D \times \partial O)^2 $ such that $v^{\e'}$ two-scale converges to $v^0$ in $\Omega\times [0,T] \times D$.
\end{theorem}

\begin{theorem}
\label{2scgrad}
Let for any $\e > 0$, $u^\e \in L^2(\Omega\times [0,T] ; H^1(D^\e)^2)$ such that 
$$\sup_\e || u^\e ||_{L^2(\Omega\times [0,T] ; L^2(D^\e)^2)} < \infty,$$
and
$$\sup_\e \e || \nabla u^\e ||_{L^2(\Omega\times [0,T] ; L^2(D^\e)^{2\times 2})} < \infty.$$
Then
$$\sup_\e \sqrt{\e}|| u^\e ||_{L^2(\Omega\times [0,T] ; L^2(\partial O^\e)^2)} < \infty.$$
Also, if we extend by $0$ inside $O^\e$ the sequences $(u^\e)_{\e >0}$ and $(\nabla u^\e)_{\e >0}$, and denote these extensions by $\widetilde{u^\e}$ and $\widetilde{\nabla u^\e}$, then there exists $u^0 \in L^2(\Omega\times [0,T]; L^2(D;H_\#^1(Y^*)^2)$ such that
$$\widetilde{u}^\e (\omega,t, x) \overset{2-s}{\longrightarrow} u^0 (\omega,t,x,y) \mathds{1}_{Y^*}(y),$$
$$\e \widetilde{\nabla u}^\e (\omega,t, x) \overset{2-s}{\longrightarrow} \nabla_y u^0(\omega,t,x,y) \mathds{1}_{Y^*}(y),$$
and
$$u^\e_{|\partial O^\e} (\omega,t, x) \overset{2-s}{\longrightarrow}  u^0_{|\partial O}(\omega,t,x,y).$$
\end{theorem}
\begin{proof}
See \cite{A95}.
\end{proof}
\begin{theorem}
\label{2scint}
Assume that the sequence $u^\e$ two scale converges to $u\in L^2(\Omega\times[0,T] \times D\times Y)^2$. Then the sequence $U^\e(\omega,t,x)$ defined by
$$U^\e(\omega,t,x) = \int_0^t u^\e(\omega, s, x) ds$$
two scale converges to $\displaystyle\int_0^t u(\omega, s, x,y) ds$.
\end{theorem}
\begin{proof}
We have for the sequence $u^\e$ that
\begin{equation*}
\lim_{\e \to 0} \int_\Omega\int_0^T \int _D u^\e (\omega,t,x) \Psi(\omega,t,x,\frac{x}{\e}) dx dt d\mathbb{P} =\int_\Omega\int_0^T\int_D\int_Y u(\omega,t,x,y) \Psi(\omega,t,x,y) dy dx dt d\mathbb{P},
\end{equation*}
for every $\Psi \in L^2(\Omega\times [0,T]\times D; C_{\#} (Y))^2$. Now if we choose $\Psi$ to be of the form
$$\Psi (\omega,s,x,y) = \mathds{1}_{[0,t]}(s)\Psi_1(\omega,x,y),$$
we obtain that
\begin{equation*}
\lim_{\e \to 0} \int_\Omega\int_0^t \int _D u^\e (\omega,s,x) \Psi_1(\omega,x,\frac{x}{\e}) dx d\mathbb{P} =\int_\Omega\int_0^t\int_D\int_Y u(\omega,s,x,y) \Psi_1(\omega,x,y) dy dx d\mathbb{P},
\end{equation*}
for any fixed $t\in[0,T]$ and any $\Psi_1 \in L^2(\Omega \times D; C_{\#}(Y))^2$. As a consequence, for $\Psi_2 \in L^2(\Omega\times [0,T]\times D; C_{\#} (Y))^2$ the sequence
$$h^\e(t) = \int_\Omega\int_0^t \int _D u^\e (\omega,s,x) \Psi_2(\omega,t,x,\frac{x}{\e}) dx ds d\mathbb{P} = \int_\Omega \int _D U^\e (\omega,t,x) \Psi_2(\omega,t,x,\frac{x}{\e}) dx d\mathbb{P}$$
is convergent for almost every $t\in[0,T]$ to
$$h(t) = \int_\Omega\int_0^t\int_D\int_Y u(\omega,s,x,y) \Psi_2(\omega,t,x,y) dy dx ds d\mathbb{P} = \int_\Omega\int_D\int_Y U(\omega,t,x,y) \Psi_2(\omega,t,x,y) dy dx d\mathbb{P}.$$
By applying H{\"o}lder's inequality,
$$h^\e(t) \leq ||u^\e||_{L^2(\Omega\times[0,T]\times D)^2} ||\Psi_2(t)||_{L^2(\Omega\times D \times Y)^2}.$$
We use the dominated convergence theorem to deduce that
$$\int_0^T h^\e(t) dt \to \int_0^T h(t) dt.$$
\end{proof}

\section{The convergence to the macroscopic problem}
\label{section5}
In this section we will use the variational formulation \eqref{varu} for $u$ to pass to the limit in the two scale convergence sense. For this, we need first to obtain the required estimates.

\subsection{Passage to the limit}

As a consequence of the assumption \eqref{u0}, we apply theorems \ref{2scex}, \ref {2scex1}  and \ref{2scgrad} to obtain the two-scale convergence
\begin{equation}
\label{2s1}
\widetilde{u}_0^\e (\omega,x) \overset{2-s}{\longrightarrow} u_0 (\omega,x,y)\ \ \ \mbox{ in } \Omega\times D,
\end{equation}
\begin{equation}
\label{2s11}
v_0^\e (\omega,x) \overset{2-s}{\longrightarrow} v_0 (\omega,x,y)\ \ \ \mbox{ in } \Omega\times D,
\end{equation}
for $u_0 \in L^2(\Omega;L^2(D;H^1_\#(Y))^2$, such that a.e. $x\in D$, $u_0(x,\cdot)\equiv 0$ in $O$, and for $v_0 =(u_0)_{\partial O}\in L^2(\Omega;L^2(D;L^2(\partial O))^2$.
From the estimates \eqref{estU} we get also that
\begin{equation}
\label{2s2}
\widetilde{u}^\e (\omega,t,x) \overset{2-s}{\longrightarrow} u^*(\omega,t,x,y)\ \ \ \mbox{ in } \Omega \times [0,T] \times D,
\end{equation}
for $u^* \in L^2(\Omega \times [0,T] \times D;L^2_\#(Y)^2) \cap L^2(\Omega \times [0,T] \times D;H^1_\#(Y^*)^2)$. Also, it follows as a consequence of the properties of the sequence $u^\e$ that a.e. $\omega\in\Omega$, for every $t\in[0,T]$ and a.e. $x\in D$ we have:
\begin{equation}\label{u*1}
u^*(\omega,t,x,\cdot) =0 \mbox{ in } O,
\end{equation}
\begin{equation}\label{u*2}
\operatorname{div}_y u^*(\omega,t,x,\cdot) = 0 \mbox{ in } Y,
\end{equation}
and a.e. $\omega\in\Omega$ and for every $t\in[0,T]$:
\begin{equation}\label{u*3}
\operatorname{div}_x \left( \int_Y u^*(\omega,t,x,y) dy\right) =0 \mbox{ in } D.
\end{equation}

\begin{equation}
\label{2s3}
\e \widetilde{\nabla u}^\e (\omega,t,x) \overset{2-s}{\longrightarrow} \xi(\omega,t,x,y)\ \ \ \mbox{ in } \Omega \times [0,T] \times D,
\end{equation}
for $\xi \in L^2(\Omega \times [0,T] \times D;L^2_\#(Y)^{2\times 2})$ such that
\begin{equation}
\xi(\omega,t,x,y)=
\left\{
\begin{array}{ll}
\nabla_y u^*(\omega,t,x,y)& \mbox{ for } y\in Y^*\\
0& \mbox{ for } y\in O
\end{array}
\right.
\end{equation}

\begin{equation}
\label{2s4}
\widetilde{U}^\e (\omega,t,x)ds \overset{2-s}{\longrightarrow} \int_0^t u^*(\omega,s,x,y)\ \ \ \mbox{ in } \Omega \times [0,T] \times D,
\end{equation}

\begin{equation}
\label{2s5}
\e \widetilde{\nabla U}^\e (\omega,t,x) ds \overset{2-s}{\longrightarrow} \int_0^t \xi(\omega,s,x,y) ds\ \ \ \mbox{ in } \Omega \times [0,T] \times D.,
\end{equation}
where $U(\omega,t,x)=\int_0^t u(\omega,s,x) ds$.
In the variational formulation \eqref{varu}, we use a test function $\phi (\omega,t,x,\dfrac{x}{\e})$ of the following form $\phi (\omega,t,x,y) = \phi_1 (\omega) \phi_2 (t) \phi_3(x,\dfrac{x}{\e})$, with $\phi_1 \in L^\infty(\Omega)$, $\phi_2 \in C_0^\infty(0,T)$ and $\phi_3 \in C_0^\infty(D;C^\infty_{\#}(Y^*))^2$ such that:
$$\phi_3 (x,\cdot) \equiv 0\mbox{ in } O,$$
$$\operatorname{div}_y \phi_3 (x,\cdot) \equiv 0 \mbox{ in } Y,$$
$$\operatorname{div}_x \left( \int_Y \phi_3 (x,y)dy\right) \equiv 0 \mbox{ in } D,$$
integrate with respect to $\omega\in\Omega$ and $t\in [0,T]$ and obtain:

\begin{equation}\label{varfor3}
\begin{split}
&\int_\Omega\int_0^T\int_{D^\e} \left(u^\e(\omega,t,x)-u^\e_0(\omega,x)-\int_{0}^{t}f(s,x)-\int_{0}^{t}g_{1}(s)dW_{1}(s)\right) \phi(\omega,t,x,\dfrac{x}{\e}) dxdt d\mathbb{P}\\
+&\int_\Omega\int_0^T\int_{D^\e}\int_0^t\nu \e^2\nabla u^\e(\omega,s,x) \nabla\phi (\omega,t,x,\dfrac{x}{\e}) dsdxdtd\mathbb{P}\\
=&\int_\Omega\int_0^T\int_{\partial O^\e} \left(-\e u^\e(\omega,t,x)+\e v^\e_0(\omega,x) +\int_{0}^{t} \e g_{2}^\e(s)dW_{2}(s) ds 
\right) \phi(\omega,t,x,\dfrac{x}{\e}) d\sigma dt d\mathbb{P}\\
-& \int_\Omega\int_0^T\int_{\partial O^\e}\int_0^t \e \alpha(\dfrac{x}{\e}) u^\e(\omega,s,x)  \phi(\omega,t,x,\dfrac{x}{\e})  d\sigma dt d\mathbb{P}.
\end{split}
\end{equation}
In \eqref{varfor3} we pass to the limit term by term. From \eqref{2s2}, \eqref{2s1} and \eqref{2s11}we get
\begin{equation}
\label{lim1}
\lim_{\e\to 0} \int_\Omega\int_0^T\int_{D^\e} u^\e(\omega,t,x) \phi(\omega,t,x,\dfrac{x}{\e}) dxdt d\mathbb{P} =\int_\Omega\int_0^T\int_{D}\int_Y u^* (\omega,t,x,y) \phi(\omega,t,x,y) dy dx dt d\mathbb{P},
\end{equation}
\begin{equation}
\label{lim2}
\lim_{\e\to 0} \int_\Omega\int_0^T\int_{D^\e} u^\e_0(\omega,x) \phi(\omega,t,x,\dfrac{x}{\e}) dxdt d\mathbb{P} =\int_\Omega\int_0^T\int_{D}\int_Y u_0 (\omega,x,y) \phi(\omega,t,x,y) dy dx dt d\mathbb{P},
\end{equation}
and
\begin{equation}
\label{lim7}
\lim_{\e\to 0} \int_\Omega\int_0^T\int_{\partial O^\e} \e v^\e_0(\omega,x) \phi(\omega,t,x,\dfrac{x}{\e}) d\sigma dt d\mathbb{P} = \int_\Omega \int_0^T \int_{\partial O} v_0(\omega,t,x,y) \phi(\omega,t,x,y) d\sigma(y) dtd\mathbb{P}
\end{equation}
We use formula (5.5) from \cite{A-2s} for a fixed fix $\omega\in\Omega$ and $t\in[0,T]$ to get that the sequence of integrals $$\int_{D^\e} \int_{0}^{t}f(s,x) \phi(\omega,t,x,\dfrac{x}{\e}) dx$$ converges to $\displaystyle\int_{D}\int_{Y^*} \int_{0}^{t}f(s,x) ds \phi(\omega,t,x,y) dy dx$. We apply now Vitali's theorem to obtain:
\begin{equation}
\label{lim3}
\begin{split}
&\lim_{\e\to 0} \int_\Omega\int_0^T\int_{D^\e} \int_{0}^{t}f(s,x)ds \phi(\omega,t,x,\dfrac{x}{\e}) dxdt d\mathbb{P} \\
&=\int_\Omega\int_0^T\int_{D}\int_{Y^*} \int_{0}^{t}f(s,x) ds \phi(\omega,t,x,y) dy dx dt d\mathbb{P}.
\end{split}
\end{equation}

We use \eqref{2s5} to get
\begin{equation}
\label{lim5}
\begin{split}
&\lim_{\e\to 0} \int_\Omega\int_0^T\int_{D^\e}\int_0^t\nu \e^2\nabla u^\e(\omega,s,x) \nabla\phi (\omega,t,x,\dfrac{x}{\e}) dsdxdtd\mathbb{P} =\\
&\lim_{\e\to 0}\int_\Omega\int_0^T\int_{D^\e}\nu \e^2\nabla u^\e(\omega,t,x) \nabla\phi (\omega,t,x,\dfrac{x}{\e}) dsdxdtd\mathbb{P}=\\
&\lim_{\e\to 0}\int_\Omega\int_0^T\int_{D^\e}\nu \e^2\nabla u^\e(\omega,t,x) \left(\nabla_x\phi (\omega,t,x,\dfrac{x}{\e}) + \dfrac{1}{\e}\nabla_y \phi (\omega,t,x,\dfrac{x}{\e})\right) dsdxdtd\mathbb{P}=\\
&\lim_{\e\to 0}\int_\Omega\int_0^T\int_{D^\e}\nu \e^2\nabla u^\e(\omega,t,x) \nabla_x\phi (\omega,t,x,\dfrac{x}{\e})+\lim_{\e\to 0}\int_\Omega\int_0^T\int_{D^\e}\nu \e\nabla u^\e(\omega,t,x) \nabla_y\phi (\omega,t,x,\dfrac{x}{\e})=\\
&\int_\Omega\int_0^T \int_D \nu \int_0^t \xi(\omega,s,x,y) ds \nabla_y\phi (\omega,t,x,\dfrac{x}{\e}).
\end{split}
\end{equation}
We will now compute the limit of the boundary integrals and the stochastic ones. The results are proven in the next Lemmas.
\begin{lemma}\label{lemmalim6}
\begin{equation}
\label{lim6}
\begin{split}
&\lim_{\e\to 0} \int_\Omega\int_0^T\int_{\partial O^\e} \e u^\e(\omega,t,x) \phi(\omega,t,x,\dfrac{x}{\e}) d\sigma dt d\mathbb{P} \\
&= \int_\Omega \int_0^T \int_D \int_{\partial O} u^*(\omega,t,x,y) \phi(\omega,t,x,y) d\sigma(y) dx dtd\mathbb{P}.
\end{split}
\end{equation}

\end{lemma}
\begin{proof}
We introduce the solution of the following problem:
\begin{equation}
\label{w1}
\left\{
\begin{array}{rll}
-\Delta w_1&= -  \dfrac{|\partial O|}{|Y^*|} &\mbox{ in }\  Y^* , \\
\displaystyle \frac{\partial w_1}{\partial n} &=1&\mbox{ on }\  \partial O, \\
\displaystyle \int_{Y^*} w_1 &=0,& \\
w_1&- Y-periodic,& \\
\end{array}
\right.
\end{equation}
define next $w_1^\e(x) = \e^2 w_1(\frac{x}{\e})$, and straightforward calculations show that $w_1^\e$ satisfies:
\begin{equation}
\label{w1e}
\left\{
\begin{array}{rll}
-\Delta w_1^\e&= -\dfrac{|\partial O|}{|Y^*|} &\mbox{ in }\  D^\e , \\
\displaystyle \frac{\partial w_1^\e}{\partial n} &=\e&\mbox{ on }\  \partial O^\e. \\
\end{array}
\right.
\end{equation}
\end{proof}
We compute next the limit as follows:
\begin{equation}\nonumber\begin{split}
&\lim_{\e\to 0} \int_\Omega\int_0^T\int_{\partial O^\e} \e u^\e(\omega,t,x) \phi(\omega,t,x,\dfrac{x}{\e}) d\sigma(x) dt d\mathbb{P}=\\
&\lim_{\e\to 0} \int_\Omega\int_0^T\int_{\partial O^\e} \displaystyle \frac{\partial w_1^\e}{\partial n}(x) u^\e(\omega,t,x) \phi(\omega,t,x,\dfrac{x}{\e}) d\sigma(x) dt d\mathbb{P}=\\
&\lim_{\e\to 0} \int_\Omega\int_0^T\int_{D^\e} \operatorname{div} \left(\nabla w_1^\e (x) u^\e(\omega,t,x) \phi(\omega,t,x,\dfrac{x}{\e})\right) dx dt d\mathbb{P}=\\
&\lim_{\e\to 0} \int_\Omega\int_0^T\int_{D^\e} \Delta w_1^\e (x) u^\e(\omega,t,x) \phi(\omega,t,x,\dfrac{x}{\e}) dxdt d\mathbb{P}+\\
&\lim_{\e\to 0} \int_\Omega\int_0^T\int_{D^\e} \nabla w_1^\e (x) \cdot \nabla \left( u^\e(\omega,t,x) \phi(\omega,t,x,\dfrac{x}{\e})\right) dx dt d\mathbb{P}.
\end{split}\end{equation}
The first limit yields, after using \eqref{2s2}:
\begin{equation}\nonumber\begin{split}
&\lim_{\e\to 0} \int_\Omega\int_0^T\int_{D^\e} \Delta w_1^\e (x) u^\e(\omega,t,x) \phi(\omega,t,x,\dfrac{x}{\e}) d\sigma(x) dt d\mathbb{P}=\\
& \lim_{\e\to 0}\int_\Omega \int_0^T \int_{D^\e} \dfrac{|\partial O|}{|Y^*|} u^\e(\omega,t,x) \phi(\omega,t,x,\dfrac{x}{\e})=\\
&\int_\Omega \int_0^T \int_D \int_Y \Delta w_1(y) u^* (\omega,t,x,y) \phi(\omega,t,x,y) dy dx dt d\mathbb{P}.
\end{split}\end{equation}
Next ,
\begin{equation}\nonumber\begin{split}
&\lim_{\e\to 0} \int_\Omega\int_0^T\int_{D^\e} \nabla w_1^\e (x) \cdot \nabla \left( u^\e(\omega,t,x) \phi(\omega,t,x,\dfrac{x}{\e})\right) dx dt d\mathbb{P}=\\
&\lim_{\e\to 0} \int_\Omega\int_0^T\int_{D^\e} \nabla w_1^\e (x) \cdot \left( \nabla u^\e(\omega,t,x) \phi(\omega,t,x,\dfrac{x}{\e}) + u^\e(\omega,t,x) \nabla \phi (\omega,t,x,\dfrac{x}{\e})\right) dx dt d\mathbb{P}=\\
&\lim_{\e\to 0} \int_\Omega\int_0^T\int_{D^\e} \nabla w_1 (\dfrac{x}{\e}) \cdot \left( \e \nabla u^\e(\omega,t,x) \phi(\omega,t,x,\dfrac{x}{\e}) + u^\e(\omega,t,x) \e \nabla \phi (\omega,t,x,\dfrac{x}{\e})\right) dx dt d\mathbb{P}.
\end{split}
\end{equation}
We use the two scale limits \eqref{2s2} and \eqref{2s3} and the fact that $w_1 \in W^{1,\infty}(Y^*)$ to obtain 
\begin{equation}\nonumber
\int_\Omega\int_0^T \int_D \int_{Y^*} \nabla w_1 (y) \cdot \left( \nabla_y u^*(\omega,t,x,y) \phi(\omega,t,x,y) + u^*(\omega,t,x,y)  \nabla_y \phi (\omega,t,x,y)\right) dy dx dt d\mathbb{P}.
\end{equation}
Adding these two results and integrating by parts, we get \eqref{lim6}.
\begin{lemma}\label{lemmalim4}
\begin{equation}
\label{lim4}
\begin{split}
&\lim_{\e\to 0} \int_\Omega\int_0^T\int_{D^\e} \int_{0}^{t}g_{1}(s)dW_{1}(s) \phi(\omega,t,x,\dfrac{x}{\e}) dxdt d\mathbb{P} =\\
&\int_\Omega\int_0^T\int_{D}\int_{Y^*} \int_{0}^{t}g_{1}(s)dW_{1}(s) \phi(\omega,t,x,y) dy dx dt d\mathbb{P}.
\end{split}
\end{equation}
\end{lemma}
\begin{proof}
We use the decomposition of $\phi$ and rewrite the difference between the two sides as 
$$\int_\Omega \phi_1(\omega)d\mathbb{P}\int_0^T\phi_2(t)\left(\int_{D^\e} \sum_{i=1}^\infty \sqrt{\lambda_{i1}}\int_{0}^{t}g_1 (s) e_{i1}(x) d\beta_i(s) \left(\phi_3(x,\dfrac{x}{\e})-\int_{Y^*}\phi_3(x,y) dy\right) dx\right)dt,$$
where $(\beta_i)_{i=1}^\infty$ is a sequence of real valued independent Brownian motions, and $(e_{i1})_{i=1}^\infty$  and  $(\lambda_{i1})_{i=1}^\infty$ are the ones defined in \eqref{eqCT}. Using H{\"o}lder's inequality it is sufficient to show that:
$$
\lim_{\e\to 0} \int_\Omega\int_0^T\left|\int_{D^\e} \sum_{i=1}^\infty \sqrt{\lambda_{i1}}\int_{0}^{t}g_1 (s) e_{i1}(x) d\beta_i(s) \left(\phi_3(x,\dfrac{x}{\e})-\int_{Y^*}\phi_3(x,y) dy\right) dx\right|^2dtd\mathbb{P} = 0.
$$
Using the stochastic Fubini theorem and then It\^{o}' isometry, the limit becomes:
$$
\lim_{\e\to 0} \int_\Omega\int_0^T\sum_{i=1}^\infty\lambda_{i1}\left|\int_{0}^{t} ds \int_{D^\e} g_1 (s) e_{i1}(x) \left(\phi_3(x,\dfrac{x}{\e})-\int_{Y^*}\phi_3(x,y) dy\right) dx\right|^2dtd\mathbb{P} = 0.
$$
We denote by $t_i^\e(t) = \left|\displaystyle\int_{0}^{t} ds \int_{D^\e} g_1 (s) e_{i1}(x) \left(\phi_3(x,\dfrac{x}{\e})-\int_{Y^*}\phi_3(x,y) dy\right) dx\right|^2$. 

The property \eqref{eqCT} implies that $t_i^\e(t)$ are uniformly bounded and using again (5.5) from \cite{A-2s} and Vitali's theorem we derive that for any $i\geq 1$
$$
\lim_{\e\to 0} \int_0^T t_i^\e(t) dt =0.
$$
As $\displaystyle \sum_{i=1}^\infty \lambda_{i1} < \infty$, the Lemma is proved.
\end{proof}
\begin{lemma}\label{lemmalim8}
\begin{equation}
\label{lim8}
\begin{split}
&\lim_{\e\to 0} \int_\Omega\int_0^T\int_{\partial O^\e} \int_{0}^{t} \e g_{2}^\e(s)dW_{2}(s) \phi(\omega,t,x,\dfrac{x}{\e}) d\sigma(x) dt d\mathbb{P} =\\
&\int_{\Omega}  \int_0^T \int_D \int_{\partial O} \int_0^t g_{21} (s) dW_2(s) \phi(\omega,t,x,y) d\sigma(y) dx dt d\mathbb{P}\\
&+ \int_\Omega \int_0^T \int_D \int_{\partial O} \int_0^t  g_{22}(s) dW_2(s) \phi (\omega,t,x,y) d\sigma(y) dxdtd\mathbb{P}.
\end{split}
\end{equation}
\end{lemma}
\begin{proof}
We will prove this Lemma using a combination of arguments used to prove \eqref{lim6} and \eqref{lim4}. We will show
\begin{equation}
\label{lim81}
\begin{split}
&\lim_{\e\to 0} \int_\Omega\int_0^T\int_{\partial O^\e} \int_{0}^{t} \e g_{21}(s)dW_{2}(s) \phi(\omega,t,x,\dfrac{x}{\e}) d\sigma(x) dt d\mathbb{P} =\\
&\int_{\Omega}  \int_0^T \int_D \int_{\partial O} \int_0^t g_{21} (s) dW_2(s) \phi(\omega,t,x,y) d\sigma(y) dx dt d\mathbb{P},
\end{split}
\end{equation}
and
\begin{equation}
\label{lim82}
\begin{split}
&\lim_{\e\to 0} \int_\Omega\int_0^T\int_{\partial O^\e} \int_{0}^{t} \e \mathcal{R}^\e g_{22}(s)dW_{2}(s) \phi(\omega,t,x,\dfrac{x}{\e}) d\sigma(x) dt d\mathbb{P} =\\
&\int_\Omega \int_0^T \int_D \int_{\partial O} \int_0^t  g_{22}(s) dW_2(s) \phi (\omega,t,x,y) d\sigma(y) dxdtd\mathbb{P}.
\end{split}
\end{equation}
To show \eqref{lim81} we use the functions $w_1$ and $w_1^\e$ defined in \eqref{w1} and \eqref{w1e} to transform the boundary integrals into integrals in the volume. We rewrite the integral as
\begin{equation}\nonumber
\begin{split}
&\int_\Omega\int_0^T\int_{\partial O^\e} \dfrac{\partial w_1^\e}{\partial n} \int_{0}^{t} g_{21}(s)dW_{2}(s) \phi(\omega,t,x,\dfrac{x}{\e}) d\sigma(x) dt d\mathbb{P}=\\
&\int_\Omega\int_0^T\int_{ D^\e} \Delta w_1^\e \int_{0}^{t} g_{21}(s)dW_{2}(s) \phi(\omega,t,x,\dfrac{x}{\e}) dx dt d\mathbb{P}+\\
&\int_\Omega\int_0^T\int_{ D^\e} \nabla w_1^\e \nabla \left(\int_{0}^{t} g_{21}(s)dW_{2}(s) \phi(\omega,t,x,\dfrac{x}{\e}) dx \right)dt d\mathbb{P}.
\end{split}
\end{equation}
As in Lemma \ref{lemmalim4}
\begin{equation}\nonumber
\begin{split}
&\lim_{\e\to 0}\int_\Omega\int_0^T\int_{ D^\e} \Delta w_1^\e \int_{0}^{t} g_{21}(s)dW_{2}(s) \phi(\omega,t,x,\dfrac{x}{\e}) dx dt d\mathbb{P}=\\
&\int_\Omega\int_0^T\int_D\int_{Y^*} \Delta w_1 \int_0^t g_{21}(s)dW_{2}(s) \phi(\omega,t,x,y) dydxdtd\mathbb{P},
\end{split}
\end{equation}
and
\begin{equation}\nonumber
\begin{split}
&\lim_{\e\to 0}\int_\Omega\int_0^T\int_{ D^\e} \nabla w_1^\e \nabla \left(\int_{0}^{t} g_{21}(s)dW_{2}(s) \phi(\omega,t,x,\dfrac{x}{\e}) dx \right)dt d\mathbb{P}=\\
&\int_\Omega\int_0^T\int_D\int_{Y^*} \nabla w_1 \int_0^t g_{21}(s)dW_{2}(s) \nabla_y\phi(\omega,t,x,y) dydxdtd\mathbb{P},
\end{split}
\end{equation}
where we used also that $w_1\in W^{1,\infty}(Y^*)$. We add these two limits and integrate by parts to obtain \eqref{lim81}.

To show \eqref{lim82} we denote by $h_i(s)=g_{22}(s) e_{i2} \in L^2(\partial O)$, 
for each $i\in \mathbb{Z}_+$ and $s\in [0,T]$. From \eqref{eqCT} we deduce that
\begin{equation}
\label{hq}
\sup_{s\in[0,T]} \sum_{i=1}^\infty \lambda_{i2} ||h_i(s)||^2_{L^2(\partial O)} < +\infty.
\end{equation}
We will define $w_i(s)$ similarly as $w_1$ to be the unique element in $H^1(Y^*)$ that solves:
\begin{equation}
\label{wis}
\left\{
\begin{array}{rll}
-\Delta w_i(s)&= - \dfrac{1}{|Y^*|} \displaystyle\int_{\partial O} h_i(s) d\sigma &\mbox{ in }\  Y^* , \\
\displaystyle \frac{\partial w_i(s)}{\partial n} &=h_i(s)&\mbox{ on }\  \partial O, \\
\displaystyle \int_{Y^*} w_i(s) &=0,& \\
w_i(s)&- Y-periodic,& \\
\end{array}
\right.
\end{equation}
and $w_i^\e(s) = \e^2 w_i(s)\left( \dfrac{\cdot}{\e} \right)$ that will solve
\begin{equation}
\label{wise}
\left\{
\begin{array}{rll}
-\Delta w_i^\e(s)&=  - \dfrac{1}{|Y^*|} \displaystyle\int_{\partial O} h_i(s) d\sigma &\mbox{ in }\  D^\e , \\
\displaystyle \frac{\partial w_i^\e(s)}{\partial n} &=\e \mathcal{R}^\e h_i (s)&\mbox{ on }\  \partial O^\e. \\
\end{array}
\right.
\end{equation}
There exists a constant $C$ independent of $i$ and $s$ such that $||w_i(s)||_{H^1(Y^*)^2} \leq C ||h_i(s)||_{L^2(\partial O)^2}$ and using \eqref{hq} we also have:
\begin{equation}
\label{propwis}
\sup_{s\in[0,T]}\sum_{i=1}^\infty \lambda_{i2}||w_i(s)||^2_{H^1(Y^*)^2}  < \infty.\\
\end{equation}
\begin{equation}
\nonumber
\begin{split}
&\int_\Omega\int_0^T\int_{\partial O^\e} \int_{0}^{t} \e \mathcal{R}^\e g_{22}(s)dW_{2}(s) \phi(\omega,t,x,\dfrac{x}{\e}) d\sigma(x) dt d\mathbb{P}=\\
&\int_\Omega\int_0^T\int_{\partial O^\e}\int_0^t \sum_{i=1}^\infty \e \sqrt{\lambda_{i2}}\mathcal{R}^\e g_{22}(s) e_{i2} d\beta_{i}(s)\phi(\omega,t,x,\dfrac{x}{\e}) d\sigma(x) dt d\mathbb{P}.
\end{split}
\end{equation}
Using the stochastic Fubini theorem, It\^{o}'s isometry, the functions $w_i(s)$ and $w_i^\e(s)$ defined in \eqref{wis} and \eqref{wise} we pass to the limit and get the result.
\end{proof}

\begin{lemma}\label{lemmalim9}
\begin{equation}
\label{lim9}
\begin{split}
&\lim_{\e\to 0} \int_\Omega\int_0^T\int_{\partial O^\e} \int_0^t \e \alpha(\dfrac{x}{\e}) u^\e(\omega,s,x) ds  \phi(\omega,t,x,\dfrac{x}{\e}) d\sigma dt d\mathbb{P} =\\
& \int_\Omega\int_0^T\int_D \int_{\partial O} \int_0^t \alpha(y) u^* (\omega,s,x,y)ds \phi(\omega,t,x,y) d\sigma(y) dx dt d\mathbb{P}.
\end{split}
\end{equation}
\end{lemma}
\begin{proof}
The proof works in the same way as for the Lemma \ref{lemmalim6}, except we need to use instead of $w_1$, the function $w_\alpha$ define as the solution of:
\begin{equation}
\label{wb}
\left\{
\begin{array}{rll}
-\Delta w_\alpha&= -  \dfrac{\int_{\partial O} \alpha(y) d\sigma(y)}{|Y^*|} &\mbox{ in }\  Y^* , \\
\displaystyle \frac{\partial w_\alpha}{\partial n} &=\alpha&\mbox{ on }\  \partial O, \\
\displaystyle \int_{Y^*} w_\alpha &=0,& \\
w_\alpha&- Y-periodic.& \\
\end{array}
\right.
\end{equation}
Everything follows similarly.
\end{proof}

\subsection{The variational formulation of the two-scale limit}
 
We get, by passing to the limit,  the following equation:
\begin{equation}\label{limvarfor}
\begin{split}
&\int_\Omega\int_0^T\int_{D}\int_Y \left( u^* (\omega,t,x,y) -u_0(\omega,x,y)-\int_{0}^{t}f(s,x)-\int_{0}^{t}g_{1}(s)dW_{1}(s)\right) \phi(\omega,t,x,y) dy dx dt d\mathbb{P}\\
+&\int_\Omega\int_0^T\int_{D}\int_Y \int_0^t\nu\nabla_y u^*(\omega,s,x,y) \nabla_y\phi (\omega,t,x,y) dsdydxdtd\mathbb{P}\\
=&\int_\Omega \int_0^T \int_D \int_{\partial O} \left( u^*(\omega,t,x,y) -v_0 (\omega,x,y) -\int_0^t \alpha(y) u^*(\omega,s,x,y) \right)\phi(\omega,t,x,y) d\sigma(y) dy dx dtd\mathbb{P}\\
+&\int_{\Omega}  \int_0^T \int_D \int_{\partial O} \int_0^t g_{21} (s) dW_2(s) \phi(\omega,t,x,y) d\sigma(y) dx dt d\mathbb{P}\\
+& \int_\Omega \int_0^T \int_D \int_{\partial O} \int_0^t  g_{22}(s) dW_2(s) \phi (\omega,t,x,y) d\sigma(y) dxdtd\mathbb{P},
\end{split}
\end{equation}
for every $\phi (\omega,t,x, y)$ of the following form $\phi (\omega,t,x,y) = \phi_1 (\omega) \phi_2 (t) \phi_3(x, y)$, with $\phi_1 \in L^\infty(\Omega)$, $\phi_2 \in C_0^\infty(0,T)$ and $\phi_3 \in C^\infty_0(D;C_{\#}^\infty (Y^*)^2$ such that:
$$\phi_3 (x,\cdot) \equiv 0\mbox{ in } Y,$$
$$\operatorname{div}_y \phi_3 (x,\cdot) \equiv 0 \mbox{ in } Y,$$
$$\operatorname{div}_x \left( \int_Y \phi_3 (x,y)dy\right) \equiv 0 \mbox{ in } D.$$
So for every $t\in[0,T]$, and a.e. $\omega\in\Omega$ and for every $\phi_3 \in  C^\infty_0(D;C_{\#}^\infty (Y^*))^2$ with the above properties we get the equation:
\begin{equation}\label{2svarfor}
\begin{split}
&\int_{D}\int_{Y^*} \left( u^* (\omega,t,x,y) -u_0(\omega,x,y)-\int_{0}^{t}f(s,x)-\int_{0}^{t}g_{1}(s)dW_{1}(s)\right) \phi_3(x,y) dy dx \\
+&\int_{D}\int_{Y^*} \int_0^t\nu\nabla_y u^*(\omega,s,x,y) \nabla_y\phi_3(x,y) dsdydx\\
=&\int_D \int_{\partial O} \left( -u^*(\omega,t,x,y) +v_0 (\omega,x,y) -\int_0^t \alpha(y) u^*(\omega,s,x,y) \right)\phi_3(x,y) d\sigma(y) dy dx\\
+&\int_D \int_{\partial O} \int_0^t g_{21} (s) dW_2(s) \phi_3(x,y) d\sigma(y) dx + \int_D \int_{\partial O} \int_0^t  g_{22}(s) dW_2(s) \phi_3(x,y) d\sigma(y) dx.
\end{split}
\end{equation}

\section{The homogenized problem}
\label{section6}

\subsection{The study of the two-scale limit equation}\label{41}
To study the equation \eqref{2svarfor} we will introduce the following spaces:

\begin{equation}
\label{spaceL}
\mf{L}^2:=L^2(D;L^2(Y^*))^2\times L^2(D;L^2(\partial O))^2,
\end{equation}

\begin{equation}
\label{spaceH}
\mf{H}^1:=L^2(D;H^1(Y^*))^2\times L^2(D;H^{1/2}(\partial O))^2,
\end{equation}
with the inner products 

\begin{equation}
  \langle \mf{U},\mf{V}\rangle_{\mf{L}^2}=\int_D \int_{Y^*}\big[  u(x,y) \cdot v(x,y)\big] dydx
  +\int_D\int_{\partial O} \big[  \o{u}(x,y') \cdot \o{v}(x.y')\big] d\sigma(y') dx,
\end{equation} 
and 
\[
 ((\mf{U},\mf{V}))_{\mf{H}^1} = \int_D\int_{Y^*} \big[ \nabla_y u(x,y) \cdot \nabla_y v(x,y)\big] dy dx.\]
Let $\mf{H}$ be the closure of $\mathcal{V}$ in $\mf{L}^2$, and $\mf{V}$ the closure of $\mathcal{V}$ in $\mf{H}^1$,  
where 
\begin{equation}
\label{spaceV}
\begin{split}
\mathcal{V}:=&\left\{ \mf{U}=\left(\begin{array}{c} 
u\\
\o{u}
\end{array}\right)
\in C_0^\infty(D;C^\infty_{\#}(Y^*))^2\times \gamma^Y(C_0^\infty(D;C^\infty_{\#}(Y^*))^2)\ \ |\ \right. \\
&\ \left. \o{u}=\gamma^Y(u), \ \operatorname{ div }_y u=0,\ \o{u}\cdot n_y=0,\ \operatorname{ div }_x\left(\int_{Y^* }u dy\right)=0, \left(\int_{Y^* }u dy\right)\cdot n =0 \mbox{ on } \partial D\right\}.
\end{split}
\end{equation}
By $\gamma^Y$ we denoted the trace operator on $\partial O$ and by $n_Y$ the normal vector on $\partial O$ pointing inside $O$.
We denote by $\Pi$ the projection operator from $\mf{L}^2$ onto $L^2(D;L^2(Y^*))^2$ and let $H=\Pi \mf{H}$, $V=\Pi\mf{V}$. We have:
\begin{equation}\label{H}
\begin{split}
H=\{ u\in L^2(D; L^2_\#(Y))^2 | &\operatorname{ div }_x \left( \int_Y u dy\right) =0,\ \left( \int_Y u dy\right)\cdot n=0 \mbox{ on } \partial D,\ \operatorname{ div }_y u=0,\\
&\o{u}\cdot n_y=0 \mbox{ on } \partial O \mbox{ for some } \o{u} \in L^2(D; L^2(\partial O))^2\},
\end{split}
\end{equation}
and
\begin{equation}\label{V}
\begin{split}
V=\{ u\in L^2(D; H^1_\#(Y))^2 | &\operatorname{ div }_x \left( \int_Y u dy\right) =0,\left( \int_Y u dy\right)\cdot n=0 \mbox{ on } \partial D,\\
&\operatorname{ div }_y u=0,u\cdot n_y=0 \mbox{ on } \partial O\}.
\end{split}
\end{equation}
The spaces defined will help us rewrite the equation for $u^*$, \eqref{2svarfor} as a stochastic partial differential equation in a product space that we will be able to study. First we formulate the stochastic partial differential equation satisfied by $u^*$. For this we need an orthogonality result that we will prove in the following lemma:
\begin{lemma} \label{ortog1}
Let $\mf{V}_p$ be the closure of the space 
\begin{equation}
\label{spaceVp}
\left\{ \mf{U}=\left(\begin{array}{c} 
u\\
\o{u}
\end{array}\right)
\in C_0^\infty(D;C^\infty_{\#}(Y^*))\times \gamma^Y(C_0^\infty(D;C^\infty_{\#}(Y^*)))\ \ |\ \right. \\
\ \left. \o{u}=\gamma^Y(u), \right\},
\end{equation}
in  $L^2(D;L^2_\#(Y^*)) \times L^2(D;L^2(\partial O))$, and let $V_p = \Pi \mf{V}_p$.

Any $p \in V_p$ has a trace $\o{p} \in L^2(D;L^2(\partial O))$, and we can define $\nabla_y p \in \left(L^2(D; H^1_\#(Y^*))^2 \right)'$
\begin{equation}\label{defgrad}
\langle \nabla_y p , \phi \rangle_{\langle\left(L^2(D; H^1_\#(Y^*))^2 \right)' , L^2(D; H^1_\#(Y^*))^2\rangle} = - \int_D\int_{Y^*} p \operatorname{div} \phi dy dx + \int_D \int_{\partial O} \o{p} \phi\cdot n_y d \sigma(y) dx.
\end{equation}
Then, the operator $\nabla_y : V_p \to \left(L^2(D; H^1_\#(Y^*))^2 \right)'$ defined in \eqref{defgrad} has closed range.
\begin{proof}
Let $\nabla_y p_n$ be a sequence in $\left(L^2(D; H^1_\#(Y^*))^2 \right)'$ converging strongly to $f$. The sequence restricted to $L^2(D; H^1_{0\#}(Y^*))^2$ will also converge strongly in $\left(L^2(D; H^1_{0\#}(Y^*))^2 \right)'$, which implies (see \cite{temam}) that there exists $p \in L^2(D;L^2_\#(Y^*))$ such that
$$ \int_D\int_{Y^*} -p \operatorname{div}_y \phi = \langle f , \phi \rangle_{\langle\left(L^2(D; H^1_{0\#}(Y^*))^2 \right)' , L^2(D; H^1_{0\#}(Y^*))^2\rangle} ,$$
and $p_n$ converges strongly to $p$ in $L^2(D;L^2_\#(Y))$. It follows that for every $\phi \in L^2(D;H^1_\#(Y^*))^2$ the sequence $\int_D \int_{\partial O} \o{p}_n \phi\cdot n_y d \sigma(y) dx$ is convergent, so by the uniform boundedness theorem the sequence $\o{p}_n$ is bounded in $L^2(D;L^2(\partial O))$. So up to a subsequence $p_n$ will converge weakly in $V_p$, and the limit is $p$. This means that $p \in V_p$ and $\nabla_y p =f$.
\end{proof}

\end{lemma}

\begin{lemma}
\label{ortog2}
The orthogonal space of $V \subset L^2(D;L^2_\#(Y^*))^2$ can be written:
\begin{equation}
\label{ortogs}
V^{\perp}=\{\nabla_x p_0(x) + \nabla_y p_1(x,y) \  | \  p_0 \in H^1(D) \mbox{ and } p_1 \in V_p\}.
\end{equation}
\end{lemma}
\begin{proof}
We will follow closely the idea from \cite{A92}, Lemma 3.8.
We will write $V=V^1 \cap V^2$, where
\begin{equation}\label{V1}
V^1=\{ u\in L^2(D; H^1_\#(Y))^2 | \operatorname{ div }_x \left( \int_Y u dy\right) =0,\ \left( \int_Y u dy\right)\cdot n=0 \mbox{ on } \partial D\},
\end{equation}
and
\begin{equation}\label{V2}
V^2=\{ u\in L^2(D; H^1_\#(Y))^2 | \operatorname{ div }_y u=0,\ u\cdot n_y=0 \mbox{ on } \partial O\}.
\end{equation}
We will show first that 
\begin{equation}\label{perp1}
(V^1)^\perp = \{\nabla_x p_0\ | \ p_0 \in L^2(D)\},
\end{equation}
and
\begin{equation}\label{perp2}
(V^2)^\perp = \{\nabla_y p_1\ | \ p_1 \in V_p\}.
\end{equation}
Let us define the operator $\operatorname{Div}_y : L^2(D; H^1_\#(Y^*))^2 \to \left(V_p\right)'$, by
\begin{equation}\label{defDiv}
\langle\operatorname{Div}_y u , p\rangle_{\langle \left(V_p\right)' , V_p \rangle} = -\int_D\int_{Y^*} \operatorname{div}_y u p dy dx + \int_D \int_{\partial O} u\cdot n_y \o{p} d\sigma(y) dx.
\end{equation}
It is easy to see that $\nabla_y : V_p \to \left(L^2(D; H^1_\#(Y^*))^2 \right)'$ defined in \eqref{defgrad} is its adjoint. The operator $\nabla_y$ has closed range according to Lemma \ref{ortog1}, which implies that its range equal to $(Ker \operatorname{Div}_y)^\perp = (V^2)^\perp$, which shows \eqref{perp2}. Equality \eqref{perp1} is shown in a similar way.
But $V^\perp = (V^1 \cap V^2)^\perp =\o{ (V^1)^\perp + (V^2)^\perp}$, so we need to show that $(V^1)^\perp + (V^2)^\perp$ is closed or equivalently that $V^1 + V^2$ is closed. We will show exactly as in \cite{A92} that $V^1 + V^2 = L^2(D;H^1_\#(Y^*))^2$.
We define the solutions of the cell problems in $Y^*$:
\begin{equation}\label{cell1}
\left\{
\begin{array}{rll}
-\Delta w_i + \nabla q_i&= e_i \mbox{ in }\    Y^*, \\
\operatorname{div} w_i &=0\ \mbox{  in }\  Y^*, \\
w_i&=0\ \mbox{  on }\ \partial O, \\
w_i & Y-periodic,
\end{array}
\right.
\end{equation}
for $e_i$ the unit vector, $i\in\{1,2\}$. The matrix with the entries $$M_{ij} = \int_{Y^*} (w_i)_j dy = \int_{Y^*} \nabla w_i \cdot \nabla w_j dy,$$
is symmetric and positive definite, so for any $u\in L^2(D;H^1_\#(Y^*))^2$ there exists a solution $v\in H^1(D)$, unique up to an additive constant of the elliptic problem
\begin{equation}\label{prob}
\left\{
\begin{array}{rll}
-\operatorname{div} (M\nabla v - \int_{Y^*} u dy ) &= 0 \mbox{ in }\ D, \\
(M \nabla v - \int_{Y^*} u dy ) \cdot n &=0\mbox{ on }\ \partial D.
\end{array}
\right.
\end{equation}
As $\sum_{i=1}^2 w_i(y) \dfrac{\partial v} {\partial x_i}(x) \in V^2$ and $u(x,y) - \sum_{i=1}^2 w_i(y) \dfrac{\partial v} {\partial x_i}(x) \in V^1$, then $L^2(D;H^1_\#(Y^*))^2 = V^1 + V^2$.
\end{proof}
The variational formulation \eqref{2svarfor} and Lemma \ref{ortog2} allows us to formulate the two-scale stochastic partial differential equation satisfied by $u^*$, the two scale limit of the sequence $u^\e$ as:
\begin{equation}
\label{2ssystem}
\left\{
\begin{array}{rll}
d u^*(t,x,y) &=\left [ \nu\Delta_{yy} u^*(t,x,y)- \nabla_x p(t,x) - \nabla_y p_1(t,x,y)\right ]dt\\
&+f(t,x)dt +g_1(t) d W_1(t) &\mbox{in}\  [0,T]\times D\times Y^*, \\
\operatorname{div}_y u^*(t,x,y) &=0&\mbox{in}\  [0,T]\times D\times Y^*, \\
u^*(t,x,y) \cdot n_y&=0&\mbox{on}\ [0,T]\times D\times \partial O, \\
d u^*_{\tau_y} (t,x,y) &=-\left[\nu\left( \dfrac{\partial u^*(t,x,y)}{\partial n_y}\right)_{\tau_y}  + \alpha(y) u^*_{\tau_y} (t,x,y)\right] dt\\
&+g_{21}(t)_{\tau_y} d W_2(t)+g_{22}(t)_{\tau_y}d W_2(t)&\mbox{on}\  [0,T]\times D\times \partial O, \\
\operatorname{div}_x\left(\displaystyle \int_Y u^*(t,x,y)\right) &=0&\mbox{in}\ [0,T]\times D, \\
\left(\displaystyle \int_Y u^*(t,x,y)\right) \cdot n_x&=0&\mbox{on}\ [0,T]\times \partial D, \\
u^*(0,x,y)&=u_0(x,y)&\mbox{in}\ [0,T]\times D\times Y^*, \\
u^*(0,x,y)&=v_0(x,y)&\mbox{on}\  [0,T]\times D\times \partial O,\\
\end{array}
\right.
\end{equation}
Let us define the linear operator $\mf{A} : D(\mf{A}) \subset \mf{H} \to \mf{H}'$ by 
\begin{equation}
\label{defA}
\mf{A} \mf{U}=\mf{A} \left( \begin{array}{c}
u\\
\o{u}
\end{array}\right)
=
\left(\begin{array}{c}
-\nu  \Delta_{yy} u\\
\nu  \left(\dfrac{\partial u}{\partial n_y}\right)_{\tau_y}+\alpha\o{u}_{\tau_y}
\end{array}\right),
\end{equation}
with
$$D(\mf{A})=\{ \mf{U}\in \mf{V}  | -\Delta_{yy} u \in L^2(D;L^2_\#(Y^*))^2\ \mbox{ and } \left(\dfrac{\partial u}{\partial n_y}\right)_{\tau_y} \in L^2(D;L^2(\partial O))^2\},$$
and
$$\mf{A} \mf{U} \cdot \mf{V}=\displaystyle\int_{D}\int_{Y^*} \nu \nabla_y u \nabla_y v dy dx+\int_D\int_{\partial O} \alpha \overline{u}_{\tau_y} \overline{v}_{\tau_y} d \sigma(y) dx.$$

\begin{lemma}\label{lemmaA} (Properties of the operator $\mf{A}$) The linear operator $\mf{A}$ is positive and self-adjoint in $\mf{H}$.
\end{lemma}
\begin{proof}
The operator is obviously symmetric and also coercive:
$$\langle \mf{A} \mf{U}, \mf{U}\rangle = \nu \displaystyle\int_D\int_{Y^*} \nabla_y u \nabla_y udy dx+\int_D\int_{\partial O} \alpha \o{u}_{\tau_y} \o{u}_{\tau_y} d \sigma(y) dx \geq c ||\mf{U}||^2_{\mf{V}}.$$
The self-adjointness follows if we show that
$$D(\mf{A}) = \{\mf{U} \in \mf{V}\ |\  \langle \mf{A} \mf{U} ,\mf{V} \rangle \leq C ||\mf{V}||_{\mf{H}} \mbox{ for every } \mf{V} \in \mf{V}\}.$$
But
\begin{equation} 
\begin{split}
\langle \mf{A} \mf{U}, \mf{V}\rangle &\leq C ||\mf{V}||_{\mf{H}^\e}  \Longleftrightarrow\\
\nu \displaystyle\int_{D}\int_{Y^*} \nabla_y u \nabla_y v dx &\leq C||v||_{L^2(D;L^2(Y^*))^2} + C ||v||_{L^2(D;L^2(\partial O))^2}\  \Longleftrightarrow\\
\nu \displaystyle\int_{D}\int_{Y^*} -\Delta u  v dx + \displaystyle \int_D\int_{\partial O} \left(\dfrac{\partial u}{\partial n_y}\right)_{\tau_y} v_{\tau_y} d\sigma &\leq C||v||_{L^2(D;L^2(Y^*))^2} + C ||v||_{L^2(D;L^2(\partial O))^2},\\
\end{split}
\end{equation}
so
\begin{equation}
\begin{split}
&\langle \mf{A} \mf{U}, \mf{V}\rangle \leq C ||\mf{V}||_{\mf{H}^\e} \forall \mf{V} \in \mf{V}^\e \Leftrightarrow\\
&\Delta u \in L^2(D;L^2_\#(Y^*))^2 \mbox{ and } \left(\dfrac{\partial u}{\partial n_y}\right)_{\tau_y} \in L^2(D;L^2(\partial O))^2\Leftrightarrow\\
&\mf{U} \in D(\mf{A}).
\end{split}
\end{equation}
\end{proof}
We define $\mathbf{F} \in L^2(0,T;L^2(D;L^2_\#(Y^*))^2\times L^2(D;L^2(\partial O))^2$ by
\begin{equation}
\label{defF}
\mathbf{F}(t,x)=\left(\begin{array}{c}
f(t,x)\\
0
\end{array}\right),
\end{equation}
and

\begin{equation}
\label{defG}
\mf{G}(t)=\left(\begin{array}{cc}
g_1(t) & 0\\
0 &  g_{22}(t)_{\tau_y} + g_{21}(t)_{\tau_y}
\end{array}\right), \qquad  \mathbf{W}(t)=(W_{1}(t), W_{2}(t)),
\end{equation}
we rewrite the system \eqref{2ssystem} in the compact form

\begin{equation}
\label{2ssystemU}
\left\{
\begin{array}{rll}
d \mf{U}^*(t) &+\mf{A} \mf{U}^* (t) dt =\mathbf{F}(t) dt+\mathbf{G}(t) d \mathbf{W}, \\

\mf{U}^*(0)&=\mf{U}_0^*=\left( \begin{array}{c}
u_0\\
v_0
\end{array}
\right).
\end{array}
\right.
\end{equation}

\begin{theorem}\label{exU*}
Let $\mf{S}^*$ be the semigroup associated with the operator $\mf{A}$. Then,  for any $T>0$, there exists a unique mild solution $\mf{U}^*$ of the equation \eqref{2ssystemU}, $$\mf{U}^*\in L^2(\Omega; C([0,T]; \mf{H})\cap L^{2}([0,T]; \mf{V})),$$
\begin{align}\label{mildU*}
\mf{U}^*(t) = \mf{S}(t)\mf{U}^*(0) +\int_0^t \mf{S}^*(t-s)\mf{F}(s) ds +\int_0^t \mf{S}(t-s) \mf{G}(s) d\mf{W}(s)
\end{align}
which is also a weak solution in the sense:
\begin{align}\label{varU*}
\langle \mf{U}^*(t),\bm{\phi}\rangle &+\int_{0}^{t}\langle \mf{A} \mf{U}^*(s),\bm{\phi}\rangle ds=\langle \mf{U}^*_{0},\bm{\phi}\rangle+\int_{0}^{t} \langle \mf{F}(s),\bm{\phi}\rangle ds + \int_{0}^{t} \langle  \mf{G}(s)  d\mf{W}(s), \bm{\phi}\rangle
\end{align}
for all $t\in [0,T]$, a.e. $\omega\in\Omega$, and for all $\bm{\phi}\in \mf{V}$.
\end{theorem}
\begin{proof} Similar to theorem 3.1.
\end{proof}
The function $u^*$ satisfies the variational formulation \eqref{2svarfor}. From Lemma \ref{ortog2}, we derive the existence of $P_0 \in L^2(\Omega;C([0,T];H^1(D)))$ and $P_1 \in L^2(\Omega;C([0,T];V_p))$ such that:
\begin{equation}\label{2svarforPP}
\begin{split}
&\int_{D}\int_{Y^*} \left( u^* (\omega,t,x,y) -u_0(\omega,x,y)-\int_{0}^{t}f(s,x)-\int_{0}^{t}g_{1}(s)dW_{1}(s)\right) \phi(x,y) dy dx \\
+&\int_{D}\int_{Y^*} \int_0^t\nu\nabla_y u^*(\omega,s,x,y) \nabla_y\phi(x,y) dsdydx+\int_D\int_{Y^*} \nabla _x P_0(\omega,t,x) \phi(x,y)dydx\\
-&\int_D\int_{Y^*}P_1(\omega,t,x,y)\operatorname{div}_y\phi(x,y)dy dx +\int_D \int_{\partial O } P_1 (\omega,t,x,y) \phi(x,y)\cdot n_y d\sigma(y) dx\\
=&\int_D \int_{\partial O} \left( -u^*(\omega,t,x,y) +v_0 (\omega,x,y) -\int_0^t \alpha(y) u^*(\omega,s,x,y) \right)\phi(x,y) d\sigma(y) dy dx\\
+&\int_D \int_{\partial O} \int_0^t g_{21} (s) dW_2(s) \phi(x,y) d\sigma(y) dx + \int_D \int_{\partial O} \int_0^t  g_{22}(s) dW_2(s) \phi(x,y) d\sigma(y) dx,
\end{split}
\end{equation}
$\mathbb{P}$-a.s., and for every $\phi \in L^2(D;H^1_\#(Y^*))$.

An equivalent variational formulation for $u^*$ can be written using only the pressure $P_0$, but using test functions $\phi\in V^2$:
\begin{equation}\label{2svarforP}
\begin{split}
&\int_{D}\int_{Y^*} \left( u^* (\omega,t,x,y) -u_0(\omega,x,y)-\int_{0}^{t}f(s,x)-\int_{0}^{t}g_{1}(s)dW_{1}(s)\right) \phi(x,y) dy dx \\
+&\int_{D}\int_{Y^*} \int_0^t\nu\nabla_y u^*(\omega,s,x,y) \nabla_y\phi(x,y) dsdydx+\int_D\int_{Y^*} \nabla _x P_0(\omega,t,x) \phi(x,y)dydx\\
=&\int_D \int_{\partial O} \left( -u^*(\omega,t,x,y) +v_0 (\omega,x,y) -\int_0^t \alpha(y) u^*(\omega,s,x,y) \right)\phi(x,y) d\sigma(y) dy dx\\
+&\int_D \int_{\partial O} \int_0^t g_{21} (s) dW_2(s) \phi(x,y) d\sigma(y) dx + \int_D \int_{\partial O} \int_0^t  g_{22}(s) dW_2(s) \phi(x,y) d\sigma(y) dx.
\end{split}
\end{equation}

\subsection{Cell problems}\label{42}
For every $1\leq i \leq 2$, let $e_i$ be the unit vector corresponding to the $i's$ direction. Let $\{w_i^1, q_i^1\}$ be the solution of the following Stokes problem defined in $Y^*$.
\begin{equation}
\label{cell1}
\left\{
\begin{array}{rll}
d w_i^1(t,y) &=\left [ \nu\Delta_{yy} w_i^1(t,y) - \nabla_y q_i^1(t,y)+e_i \right ]dt &\mbox{in}\  [0,T]\times  Y^*, \\
\operatorname{div}_y w_i^1(t,y) &=0&\mbox{in}\  [0,T]\times  Y^*, \\
w_i^1(t,y) \cdot n_y&=0&\mbox{on}\ [0,T]\times  \partial O, \\
d (w_i^1)_{\tau_y} (t,y) &=-\left[\nu\left( \dfrac{\partial w_i^1(t,y)}{\partial n_y}\right)_{\tau_y}  + \alpha(y) (w_i^1)_{\tau_y} (t,y)\right] dt&\mbox{on}\  [0,T]\times  \partial O, \\
w_i^1(0,y)&=0&\mbox{in}\ [0,T]\times  Y^*, \\
w_i^1(0,y)&=0&\mbox{on}\  [0,T]\times  \partial O,\\
\end{array}
\right.
\end{equation}
and $\{w_i^2,q_i^2\}$ the solution of the problem:
\begin{equation}
\label{cell2}
\left\{
\begin{array}{rll}
d w_i^2(t,y) &=\left [ \nu\Delta_{yy} w_i^2(t,y) - \nabla_y q_i^2(t,y) \right ]dt &\mbox{in}\  [0,T]\times  Y^*, \\
\operatorname{div}_y w_i^2(t,y) &=0&\mbox{in}\  [0,T]\times  Y^*, \\
w_i^2(t,y) \cdot n_y&=0&\mbox{on}\ [0,T]\times  \partial O, \\
d (w_i^2)_{\tau_y} (t,y) &=-\left[\nu\left( \dfrac{\partial w_i^2(t,y)}{\partial n_y}\right)_{\tau_y}  + \alpha(y) (w_i^2)_{\tau_y} (t,y) + (e_i)_{\tau_y}\right] dt&\mbox{on}\  [0,T]\times  \partial O, \\
w_i^2(0,y)&=0&\mbox{in}\ [0,T]\times  Y^*, \\
w_i^2(0,y)&=0&\mbox{on}\  [0,T]\times  \partial O.\\
\end{array}
\right.
\end{equation}
Also let $\{w_i^3,q_i^3\}$ be the solution of the stochastic partial differential equation in $Y^*$:

\begin{equation}
\label{cell3}
\left\{
\begin{array}{rll}
d w^3(t,y) &=\left [ \nu\Delta_{yy} w^3(t,y) - \nabla_y q^3(t,y) \right ]dt &\mbox{in}\  [0,T]\times  Y^*, \\
\operatorname{div}_y w^3(t,y) &=0&\mbox{in}\  [0,T]\times  Y^*, \\
w^3(t,y) \cdot n_y&=0&\mbox{on}\ [0,T]\times  \partial O, \\
d (w^3)_{\tau_y} (t,y) &=-\left[\nu\left( \dfrac{\partial w^3(t,y)}{\partial n_y}\right)_{\tau_y}  + \alpha(y) (w^3)_{\tau_y} (t,y) \right] dt\\
 &+ \left[g_{22}\right]_{\tau_y}(t,y)d W_2 &\mbox{on}\  [0,T]\times  \partial O, \\
w^3(0,y)&=0&\mbox{in}\ [0,T]\times  Y^*, \\
w^3(0,y)&=0&\mbox{on}\  [0,T]\times  \partial O.
\end{array}
\right.
\end{equation}

We will define a functional setting to these cell problems following a similar approach introduced for  system \eqref{2ssystem}.
We denote by $\mathcal{V}^Y$ the following space:
\begin{equation}
\label{spaceVY}
\mathcal{V}^Y:=\left\{ \mf{W}=\left(\begin{array}{c} 
w\\
\o{w}
\end{array}\right)
\in C_\#^\infty\left(\overline{Y^*}\right)^2\times \gamma^Y(C_\#^\infty\left(\overline{Y^*}\right)^2)\ \ |\ \ \mbox{div }u=0,\ \o{u}=u_{\partial O},\ \overline{u}\cdot n_y =0 \mbox{ on } \partial O \right\},
\end{equation}
and by $\mf{V}^Y$ and $\mf{H}^Y$ the closure of this space in $\mf{H}^1_Y = H^1_\#(Y^*) \times H^{1/2}(\partial O)$ and in 

$\mf{L}^2_Y = L^2_\#(Y^*)\times L^2(\partial O)$. The operator associated to these equation will be denoted by $\mf{A}^Y$, where:
\begin{equation}
\label{defAY}
\mf{A}^Y \mf{W}=\mf{A}^Y \left( \begin{array}{c}
w\\
\o{w}
\end{array}\right)
=
\left(\begin{array}{c}
-\nu  \Delta w\\
\nu  \left(\dfrac{\partial w}{\partial n_y}\right)_{\tau_y}+\alpha\o{w}_{\tau_y}
\end{array}\right).
\end{equation}

Using the semigroup $\mf{S}^Y$ associated with the operator $\mf{A}^Y$, the solutions of these equations can be writen as:
\begin{equation}
\label{w1-cell}
w_i^1(t,y) = \Pi^Y\int_0^t \mf{S}^Y(t-s) 
\operatorname{Proj}_{\mf{V}^{Y}}\left(
\begin{array}{cc}
e_i\\
0
\end{array}
\right)
 ds = \Pi^Y\int_0^t \mf{S}^Y(s) 
\operatorname{Proj}_{\mf{V}^{Y}}\left(
\begin{array}{cc}
e_i\\
0
\end{array}
\right)
ds,
\end{equation}

\begin{equation}
\label{w2}
w_i^2(t,y) = \Pi^Y\int_0^t \mf{S}^Y(t-s) 
\operatorname{Proj}_{\mf{V}^{Y}}\left(
\begin{array}{cc}
&0\\
&(e_i)_{\tau_y}
\end{array}
\right)
 ds= \Pi^Y\int_0^t \mf{S}^Y(s) 
 \operatorname{Proj}_{\mf{V}^{Y}}\left(
\begin{array}{cc}
&0\\
&(e_i)_{\tau_y}
\end{array}
\right)
 ds ,
\end{equation}

\begin{equation}
\label{w3}
w_i^3(t,y) =  \Pi^Y\int_0^t \mf{S}^Y(t-s)
\operatorname{Proj}_{\mf{V}^{Y}}\left(
\begin{array}{cc}
&0\\
&\left[g_{22}\right]_{\tau_y}
\end{array}
\right)
dW_2(s),
\end{equation}
where by $\Pi^Y$ denotes the projection of $\mf{L}^2_Y$ onto the first component.
\begin{theorem}
\label{thform}
The solution of the system \eqref{2ssystem} is given by:
\begin{equation}\label{u*}
\begin{split}
u^*(t,x,y) &= w^0(t,x,y) + \sum_{i=1}^2 \int_0^t \dfrac{dw_i^1}{dt}(t-s,y)f_i(s,x) ds\\
 &+ \sum_{i=1}^2\int_0^t \dfrac{dw_i^1}{dt}(t-s,y)(g_1)_i(s) d W_1(s)-\sum_{i=1}^2\int_0^t \dfrac{d^2w_i^1}{dt^2}(t-s,y) \dfrac{\partial P_0}{\partial x_i}(s,x) ds\\ 
& + \sum_{i=1}^2\dfrac{\partial P_0}{\partial x_i}(t,x) \dfrac{dw_i^1}{dt}(0,y)\\
& + \sum_{i=1}^2\int_0^t\dfrac{dw_i^2}{dt}(t-s,y) (g_{21})_i(s) dW_2(s) + w^3(t,y),
\end{split}
\end{equation}
where $w^0(t,x,y) = \Pi^Y \mf{S}^Y(t)\left(
\begin{array}{cc}
&u_0(x,\cdot)\\
&v_0(x,\cdot)
\end{array}
\right)(y).$
\end{theorem}
\begin{proof}
It is easy to verify that $u^*$ given above satisfies $\operatorname{div}_y u^* = 0$ in $Y^*$, and $u^* \cdot n_y =0$ on $\partial O$. Thus,  to show that it is the solution for the two scale system \eqref{2ssystem} it is sufficient to show that  $u^*$ satisfies the variational formulation \eqref{2svarforP} for a.s. $\omega\in\Omega$.

We write 
$$u^*(t,x,y)=\sum_{i=1}^7 I_i(t,x,y),$$
take a test function $\phi \in V^Y = \Pi^Y \mf{V}^Y$ and compute for each term
\begin{equation}\nonumber
\begin{split}
\mathcal{L} (I_i,\phi) =&\int_{D}\int_{Y^*} \left( I_i (t,x,y) -(I_i)_0( x,y)\right) \phi(x,y) dy dx \\
&+\int_{D}\int_{Y^*} \int_0^t\nu\nabla_y I_i(s,x,y) \nabla_y\phi(x,y) dsdydx\\
&+\int_D \int_{\partial O} \left( I_i(t,x,y) -(I_i)_0 (x,y) +\int_0^t \alpha(y) I_i(s,x,y)ds \right)\phi(x,y) d\sigma(y) dx,
\end{split}
\end{equation}
where $(I_i)_0( x,y)=I_i(0, x,y)$.

By its definition, $w^0(t,x,y)$ satisfies the variational formulation
\begin{equation}\label{varforw0}
\begin{split}
&\int_{Y^*} \left( w^0 (t,x,y) -u_0(x,y)\right) \phi(y) dy +\int_{Y^*} \int_0^t\nu\nabla_y w^0(s,y) \nabla_y\phi(y) dsdy\\
+& \int_{\partial O} \left( w^0(t,x,y) -v_0 (x,y) +\int_0^t \alpha(y) w^0(s,y) \right)\phi(y) d\sigma(y) =0 ,
\end{split}
\end{equation}
for almost all $x\in D$, and for every $\phi \in V^Y$. We obtain from here that $\mathcal{L}(I_1,\phi) = 0$.

To  compute $\mathcal{L}(I_2,\phi)$, we use the variational formulation of $w_i^1$:
\begin{equation}\label{varforw1}
\begin{split}
&\int_{Y^*} \left( w_i^1 (t,y) -w_i^1(0,y)\right) \phi(x,y) dy +\int_{Y^*} \int_0^t\nu\nabla_y w_i^1(s,y) \nabla_y\phi(x,y) dsdy\\
+& \int_{\partial O} \left( w_i^1(t,y) -w_i^1 (0,y) +\int_0^t \alpha(y) w_i^1(s,y) \right)\phi(x,y) d\sigma(y) = \int_0^t \int_{Y^*}e_i \phi(x,y)dydt,
\end{split}
\end{equation}
for almost all $x\in D$, and for every $\phi \in V^Y$. We differentiate in time and take $t=t-s$ to get:
\begin{equation}\nonumber
\begin{split}
&\int_{Y^*} \dfrac{ d w_i^1}{dt} (t-s,y) \phi(x,y) dy +\int_{Y^*} \nu\nabla_y w_i^1(t-s,y) \nabla_y\phi(x,y) dy\\
+& \int_{\partial O} \left( \dfrac{ d w_i^1}{dt}(t-s,y) + \alpha(y) w_i^1(t-s,y) \right)\phi(x,y) d\sigma(y) = \int_{Y^*}e_i \phi(x,y)dy\Longleftrightarrow
\end{split}
\end{equation}
\begin{equation}\nonumber
\begin{split}
&\int_{Y^*} \dfrac{ d w_i^1}{dt} (t-s,y) \phi(x,y) dy +\int_{Y^*} \nu\int_0^{t-s}\nabla_y \dfrac{ d w_i^1}{dt}(r,y) \nabla_y\phi(x,y) drdy\\
+& \int_{\partial O} \left( \dfrac{ d w_i^1}{dt}(t-s,y) + \alpha(y) \int_0^{t-s } \dfrac{ d w_i^1}{dt}(r,y) \right)\phi(x,y) dr d\sigma(y) = \int_{Y^*}e_i \phi(x,y)dy.
\end{split}
\end{equation}
Now we multiply with $f_i(s,x)$, integrate over $D \times[0,t]$ and sum over $i\in\{1,2\}$ to get:
\begin{equation}\nonumber
\begin{split}
&\int_D\int_{Y^*}I_2(t,x,y) \phi(x,y)dydx +\sum_{i=1}^2\int_D\int_{Y^*}\int_0^t \nu\int_0^{t-s}\nabla_y \dfrac{ d w_i^1}{dt}(r,y) \nabla_y\phi(x,y) f_i(s,x) drdsdydx\\
+& \int_D \int_{\partial O}  I_2(t,x,y) \phi(x,y)dy dx + \sum_{i=1}^2\int_D \int_{\partial O} \int_0^t \alpha(y) \int_0^{t-s } \dfrac{ d w_i^1}{dt}(r,y)\phi(x,y)f_i(s,x) dr ds d\sigma(y)dx\\
=& \int_D \int_{Y^*} \int_0^t f(s,x) \phi(x,y) dsdydx.
\end{split}
\end{equation}
We now use the following lemma to derive that $\mathcal{L}(I_2,\phi) = \displaystyle\int_D \int_{Y^*} \int_0^t f(s,x) \phi(x,y) dsdydx$.
\begin{lemma}
\label{chvar}
Let $0<T$ and $a,b \in C(0,T)$. Then:
$$\int_0^t \int_0^{t-s} a(r) b(s) dr ds = \int_0^t \int_0^s a(s-r) b(r)drds,$$
for all $t\in[0,T]$.
\end{lemma}
\begin{proof}
Both sides of the equality are differentiable functions of $t$, equal to $0$ when $t=0$, and 
\begin{equation}\nonumber
\dfrac{d } {dt} \displaystyle \int_0^t \int_0^{t-s} a(r) b(s) dr ds = \int_0^t a(t-s) b(s) ds = \dfrac{d } {dt} \int_0^t \int_0^s a(s-r) b(r)drds.
\end{equation}
\end{proof}
To show that $\mathcal{L}(I_3,\phi) = \displaystyle\int_D \int_{Y^*} \int_0^t g_1(s) \phi(x,y) dW(s)dydx$ we will use similar calculations and the following stochastic version of the Lemma \ref{chvar}:
\begin{lemma}
\label{chvar'}
Let $0<T$ and $a\in C(0,T)$. Also let $g \in C(0,T ; L_Q(K,H))$ where $K$ and $H$ are two separable Hilbert spaces and $Q$ is a linear positive operator in $K$ of trace class, and let $(W(t))_{t\geq 0}$ be a $K$-valued Wiener process. 

Then for all $t\in[0,T]$ and $\mathbb{P}$ a.s.

$$\int_0^t \int_0^{t-s} a(r) g(s) dr dW(s) = \int_0^t \int_0^s a(s-r) b(r)dW(r)ds,$$
\end{lemma}
\begin{proof}
We denote $f(t,\omega) = \displaystyle \int_0^t \int_0^{t-s} a(r) g(s) dr dW(s) - \int_0^t \int_0^s a(s-r) b(r)dW(r)ds$ and we will show first that
\begin{equation}\label{est1}
\lim_{\tau \searrow 0} \dfrac{\E   \displaystyle\sup_{0\leq t_2-t_1\leq \tau}|f(t_2,\omega) - f(t_1,\omega)| } {\tau}=0.
\end{equation}
\begin{equation}\nonumber
\begin{split}
f(t_2,\omega) - f(t_1,\omega) = &\int_{t_1}^{t_2} \int_0^{t_2-s} a(r) g(s) dr dW(s) +\int_{0}^{t_1} \int_{t_1-s}^{t_2-s}  a(r) g(s) dr dW(s)\\
- &\int_{t_1}^{t_2} \int_0^s a(s-r) g(r)dW(r)ds,
\end{split}
\end{equation}
so by adding and substracting $\displaystyle\int_{0}^{t_1} \int_{t_1}^{t_2}a(t_1-s) g(s) dr dW(s)$
\begin{equation}\nonumber
\begin{split}
 f(t_2,\omega) - f(t_1,\omega) = &\int_{t_1}^{t_2} \int_0^{t_2-s} a(r) g(s) dr dW(s) +\int_{0}^{t_1} \int_{t_1}^{t_2} (a(r-s)-a(t_1-s)) g(s) dr dW(s)\\
- &\int_{t_1}^{t_2} (\int_0^s a(s-r) g(r)dW(r) - \int_0^{t_1} a(t_1-r) g(r)dW(r))ds\\
=&\int_{t_1}^{t_2} \int_0^{t_2-s} a(r) g(s) dr dW(s) +\int_{0}^{t_1} \int_{t_1}^{t_2} (a(r-s)-a(t_1-s)) g(s) dr dW(s)\\
 -&\int_{t_1}^{t_2} \int_0^{t_1} (a(s-r)-a(t_1-r)) g(r)dW(r) ds-\int_{t_1}^{t_2} \int_{t_1}^{s} a(s-r) g(r)dW(r)ds.
\end{split}
\end{equation}
From here
\begin{equation}\nonumber
\begin{split}
 |f(t_2,\omega) - f(t_1,\omega)|^2 \leq & C|\int_{0}^{t_2-t_1} \int_0^{t_2-t_1-s} a(r) g(s+t_1) dr dW(s)|^2\\ +&C|\int_{0}^{t_1} \int_{t_1}^{t_2} (a(r-s)-a(t_1-s)) g(s) dr dW(s)|^2\\
 +&C |\int_{t_1}^{t_2} \int_0^{t_1} (a(s-r)-a(t_1-r)) g(r)dW(r)ds|^2\\
+&C|\int_{t_1}^{t_2} \int_{t_1}^{s} a(s-r) g(r)dW(r)ds|^2,
\end{split}
\end{equation}
and by applying stochastic Fubini theorem:
\begin{equation}\nonumber
\begin{split}
 |f(t_2,\omega) - f(t_1,\omega)|^2 \leq & C|\int_0^{t_2-t_1-s}  \int_{0}^{t_2-t_1}  a(r) g(s+t_1) dW(s) dr|^2\\ +&C|\int_{t_1}^{t_2} \int_{0}^{t_1} (a(r-s)-a(t_1-s)) g(s)  dW(s)dr|^2 \\
 +&C |\int_{t_1}^{t_2} \int_0^{t_1} (a(s-r)-a(t_1-r)) g(r)dW(r)ds|^2\\
+&C|\int_{t_1}^{t_2} \int_{t_1}^{s} a(s-r) g(r)dW(r)ds|^2\\
\leq &C(t_2-t_1)^2\sup_{0\leq r\leq t_2-t_1} |a(r)|^2\left|\int_0^{t_2-t_1} g(s+t_1) dr dW(s)\right|^2\\
 +& C(t_2-t_1)^2 \sup_{r\in[t_1,t_2]}|\int_{0}^{t_1} (a(r-s)-a(t_1-s)) g(s)  dW(s)|^2\\
+&C (t_2-t_1)^2 \sup_{s\in[t_1,t_2]}|\int_0^{t_1} (a(s-r)-a(t_1-r)) g(r)dW(r)|^2\\
+&C(t_2-t_1)^2 \sup_{s\in[t_1,t_2]}|\int_{t_1}^{s} a(s-r) g(r)dW(r)|^2.
\end{split}
\end{equation}
Now, we estimate $\E \displaystyle\sup_{0\leq t_2-t_1\leq \tau}|f(t_2,\omega) - f(t_1,\omega)|^2$ using Burkholder-Davis-Gundy Inequality and It\^{o}'s isometry:
\begin{equation}\nonumber
\begin{split}
\E \displaystyle\sup_{0\leq t_2-t_1\leq \tau}|f(t_2,\omega) - f(t_1,\omega)|^2 
\leq & C\tau^2 \E \left( \int_0^\tau \| g(s+t_1) \|^2_Q ds \right)
+ C\tau^2 \sup_{|s_1-s_2|\leq\tau}|a(s_1)-a(s_2)|^2  \\
 \leq& C\tau^3 + C \tau^2\sup_{|s_1-s_2|\leq\tau}|a(s_1)-a(s_2)|^2.
\end{split}
\end{equation}
which shows \eqref{est1}.

We will now derive that $\mathbb{P}$ a.s.:
\begin{equation}\label{est2}
\lim_{t_2-t_1 \searrow 0} \dfrac{ |f(t_2,\omega) - f(t_1,\omega)| } {t_2-t_1}=0.
\end{equation}

We define the positive and increasing function of $t \in (0,T]$, $\alpha$ by
$$\alpha^2(t) =  \sup_{0\leq\tau\leq t}\dfrac{ \E \displaystyle\sup_{0\leq t_2-t_1 \leq \tau}|f(t_2,\omega) - f(t_1,\omega)| } {\tau},$$
and \eqref{est1} implies that 
$$\lim_{t\searrow 0} \alpha (t) =0.$$
For any $0<\e <T-t_1$ we choose a sequence  $(r_n^\e)_{n\geq 1}$, $r_n^\e \searrow 0$ such that
$$\sum_{n=1}^\infty \alpha(r_n^\e) \leq \e.$$
Let $\Omega_n^\e$ be the set:
$$\{\omega\in\Omega \ | \ \sup_{0\leq t_2-t_1 \leq r_n^\e} |f(t_2,\omega) - f(t_1,\omega)| \geq \alpha(r_n^\e)(t_2-t_1)\},$$
which by Chebychev inequality has $\mathbb{P}(\Omega_n^\e) \leq \alpha(r_n^\e).$ On the set $\Omega^\e = \Omega\setminus\bigcup_{n\geq 1} \Omega^\e_n$ we have:
$$|f(t_2,\omega)-f(t_1,\omega)| \leq \alpha(r_n^\e)|t_2-t_1|,$$
for every $t_1,t_2\in[0,T]$ such that $|t_2-t_1|\leq r_n^\e$. Since $\mathbb{P}(\Omega^\e)\geq 1-\e$ we proved \eqref{est2}.

Since \eqref{est2} implies that $\mathbb{P}$ a.s. the function $f(t,\omega)$ has the derivative $f'(t,\omega)$ equal to $0$ in $[0,T]$ and since $f(0,\omega)=0$ then $\mathbb{P}$ a.s. $f(t,\omega)=0$ for every $t\in[0,T]$, and the Lemma is proved.\end{proof}

To compute $\mathcal{L}(I_4,\phi)$ we use \eqref{varforw1}, differentiate twice with respect to $t$ to get:
\begin{equation}\nonumber
\begin{split}
&\int_{Y^*} \dfrac{ d^2 w_i^1}{dt^2} (t-s,y) \phi(x,y) dy +\int_{Y^*} \nu\nabla_y \dfrac{ d w_i^1}  { dt }(t-s,y) \nabla_y\phi(x,y) dy\\
+& \int_{\partial O} \dfrac{ d^2 w_i^1}{dt^2 }(t-s,y)\phi(x,y) d\sigma(y) + \int_{\partial O}\alpha(y) \dfrac{ d w_i^1 } {dt }(t-s,y)\phi(x,y) d\sigma(y) = 0.
\end{split}
\end{equation}
We rewrite the equation obtained as
\begin{equation}\nonumber
\begin{split}
&\int_{Y^*} \dfrac{ d^2 w_i^1}{dt^2} (t-s,y) \phi(x,y) dy +\int_{Y^*} \int_0^{t-s}\nu\nabla_y \dfrac{ d^2 w_i^1}  { dt^2 }(r,y) \nabla_y\phi(x,y)dr dy\\
+& \int_{\partial O} \dfrac{ d^2 w_i^1}{dt^2 }(t-s,y)\phi(x,y) d\sigma(y) + \int_{\partial O}\alpha(y) \int_0^{t-s}\dfrac{ d^2 w_i^1 } {dt^2 }(r,y)\phi(x,y) dr d\sigma(y)\\
+&\int_{Y^*} \nu\nabla_y \dfrac{ d w_i^1}  { dt }(0,y) \nabla_y\phi(x,y) dy+ \int_{\partial O}\alpha(y) \dfrac{ d w_i^1 } {dt }(0,y)\phi(x,y) d\sigma(y)= 0.
\end{split}
\end{equation}

We multiply with $-\dfrac { d P_0 } { dx_i } (s,x)$, integrate over $[0,t] \times D$, and apply Lemma \ref{chvar} to get that 
\begin{equation}\begin{split}
\mathcal{L}(I_4,\phi)&=\sum_{i=1}^2\int_0^t\int_D\int_{Y^*} \nu\nabla_y \dfrac{ d w_i^1}  { dt }(0,y) \dfrac { d P_0 } { dx_i } (s,x)\nabla_y\phi(x,y) dydxds\\
&+ \sum_{i=1}^2\int_0^t\int_D\int_{\partial O}\alpha(y) \dfrac{ d w_i^1 } {dt }(0,y)\dfrac { d P_0 } { dx_i } (s,x)\phi(x,y) d\sigma(y)dxdts\\
&=\sum_{i=1}^2\int_0^t\int_D\int_{Y^*} \nu\nabla_y \dfrac{ d w_i^1}  { dt }(0,y) \dfrac { d P_0 } { dx_i } (s,x)\nabla_y\phi(x,y) dydxds,
\end{split}
\end{equation}
using the system \eqref{cell1}. 

\begin{equation}
\nonumber
\begin{split}
\mathcal{L}(I_5,\phi) &= \sum_{i=1}^2\int_D\int_{Y^*} -\dfrac{ d P_0 } {d x_i }(t,x) \dfrac { d w_i^1 } { dt } (0,y) \phi(x,y) dydx\\
& + \sum_{i=1}^2\int_D\int_{\partial O} -\dfrac{ d P_0 } {d x_i }(t,x) \dfrac { d w_i^1 } { dt } (0,y) \phi(x,y) d\sigma(y) dx\\
&- \sum_{i=1}^2 \int_0^t \int_D \int_{Y^*} \nu \dfrac{ d P_0 } {d x_i }(s,x)\dfrac { d \nabla_y w_i^1 } { dt } (0,y)\nabla_y  \phi(x,y) dy dsdx\\
&- \sum_{i=1}^2 \int_0^t \int_D \int_{\partial O} \alpha(y)\dfrac{ d P_0 } {d x_i }(s,x)\dfrac { d  w_i^1 } { dt }  (0,y)  \phi(x,y) d\sigma(y) dsdx\\
&=\sum_{i=1}^2\int_D\int_{Y^*} -\dfrac{ d P_0 } {d x_i }(t,x) (e_i-\nabla_y q_i(y)) \phi(x,y) dydx\\
&- \sum_{i=1}^2 \int_0^t \int_D \int_{Y^*} \nu \dfrac{ d P_0 } {d x_i }(s,x)\dfrac { d \nabla_y w_i^1 } { dt } (0,y)\nabla_y  \phi(x,y) dy dsdx.
\end{split}
\end{equation}
This implies that
\begin{equation}
\begin{split}
\mathcal{L}(I_4,\phi)+\mathcal{L}(I_5,\phi) = -\int_D \int_{Y^*} \nabla_x P_0(t,x) \phi(x,y) dy dx.
\end{split}
\end{equation}
The fact that $\mathcal{L}(I_6,\phi) = \displaystyle\int_D \int_{Y^*} \int_0^t g_{21}(s) dW_2(s) \phi(x,y) dy dx$ is done similarly to  $\mathcal{L}(I_3,\phi)$ and by applying lemma \eqref{chvar'}.

From the variational formulation for $w^3$ we infer that $\mathcal{L}(I_7,\phi) = \displaystyle\int_D \int_{Y^*} \int_0^t g_{22}(s) dW_2(s) \phi(x,y) dy dx$. Summing all these equation we get that $u^*$ given by \eqref{u*} is the solution of \eqref{2svarforP}.
\end{proof}
\section{Main results and concluding remarks}

\subsection{Main results} Let us summarize the main results of this paper obtained and proved  in Section 5 and Section 6.

\begin{theorem}\label{main1}
The extension of process $U^\e= \left( \begin{array}{c}
u^\e\\
\o{u}^\e
\end{array}\right)$ 
solution of system \eqref{systemu} two-scale converges to the  process $U^*=\left( \begin{array}{c}
u^*\\
\o{u}^*
\end{array}\right)$. Moreover, there exists $P_0 \in L^2(\Omega;C([0,T];H^1(D)))$ such that
$\mathbb{P}$-a.s., and for every $\phi\in V^2$, the process
$u^*$ satisfies the variational formulation previously defined in Section 6, that is 

\begin{equation}\label{2svarforP-bis}
\begin{split}
&\int_{D}\int_{Y^*} \left( u^* (\omega,t,x,y) -u_0(\omega,x,y)-\int_{0}^{t}f(s,x)-\int_{0}^{t}g_{1}(s)dW_{1}(s)\right) \phi(x,y) dy dx \\
+&\int_{D}\int_{Y^*} \int_0^t\nu\nabla_y u^*(\omega,s,x,y) \nabla_y\phi(x,y) dsdydx+\int_D\int_{Y^*} \nabla _x P_0(\omega,t,x) \phi(x,y)dydx\\
=&\int_D \int_{\partial O} \left( -u^*(\omega,t,x,y) +v_0 (\omega,x,y) -\int_0^t \alpha(y) u^*(\omega,s,x,y) \right)\phi(x,y) d\sigma(y) dy dx\\
+&\int_D \int_{\partial O} \int_0^t g_{21} (s) dW_2(s) \phi(x,y) d\sigma(y) dx + \int_D \int_{\partial O} \int_0^t  g_{22}(s) dW_2(s) \phi(x,y) d\sigma(y) dx.
\end{split}
\end{equation}
\end{theorem}
\begin{proof}
See section 5 and section 6.
\end{proof}

Let us define now the 2 by 2 matrices $K_1$ and $K_2$ with the entries $$(K_1)_{ij} = \int_{Y^*} \dfrac{d (w^1_i)_j} {dt} (t) dy = \int_{Y^*} S(t) e_i e_j dy, $$ and $$(K_2)_{ij} = \int_{Y^*} \dfrac{d (w^2_i)_j} {dt} (t) dy.$$ 
Then, we have the  following Darcy's law with memory associated with the weak limit of the $u^\e$:
\begin{theorem}\label{main2}
The process $u^\e$ converges weakly in $L^2(\Omega\times [0,T]\times D)^2$ to 
$\displaystyle\int_{Y^*} u^*(\omega,t,x,y) dy $. Moreover, $\mathbb{P}$-a.s

\begin{equation}\label{u*'}
\begin{split}
\int_{Y^*}u^*(t,x,y) dy  &= \int_{Y^*} w^0(t,x,y)dy + \int_0^t K_1(t-s) f(s,x) ds+ \int_0^t K_1(t-s)g_1(s) d W_1(s)\\
 &-\int_0^t K_1'(t-s) \nabla_x P_0(s,x) ds- K_1(0)\nabla_x P_0(t,x)+ \int_0^t K_2(t-s)g_{21}(s) dW_2(s)\\
 &+ \int_{Y^*}w^3(t,y).
\end{split}
\end{equation}
\end{theorem}
\begin{proof}
The weak convergence is a consequence of the 2-scale convergence and corollary \ref{2sc-w}. Then, it is enough to integrate \eqref{u*} over $Y^*$ and use the previous matrices $K_1$ and $K_2$ in order to get \eqref{u*'}. 
\end{proof}

\subsection{Concluding remarks}
Our main result is given in theorem \ref{thform} and by formula \eqref{main2}. The solution of the homogenized problem is a  Darcy's law with memory with two permeabilities and an extra term $w^3$ related to the  initial stochastic perturbation on the boundary of the porous medium. 
The formula \eqref{u*'} shows that in the absence of external forces, a motion of  the fluid appears. This motion is a consequence of the stochastic perturbation on the porous medium only. A similar result has been proven in \cite{CDE96} for a steady viscous linear fluid flow in porous medium with a  non homogenous slip boundary condition, where the non homogenous term is due to an electrical field with a surface distribution that generates a double layer which allows the fluid to slip. In this paper, we obtain by a rigorous homogenization procedure a modified Darcy's law that takes into account the extra term that allows the fluid to slip.

\section*{Acknowledgements}
Hakima Bessaih would like to thank the Numerical Porous Media (NumPor) SRI Center at KAUST where part of this project was started.

H. B~ was partially supported by the Simons Foundation grant \#283308,   the NSF grants DMS-1416689 and DMS-1418838.

\end{document}